\documentclass[11pt]{article}

\usepackage{mathpazo}
\usepackage{amsfonts}
\usepackage{amsmath}
\usepackage{graphicx}
\usepackage{latexsym}
\usepackage{mathabx}
\usepackage[margin=1.05in]{geometry}
  
\usepackage[T1]{fontenc}
\usepackage{times}
\usepackage{color,graphicx}
\usepackage{array}
\usepackage{enumerate}
\usepackage{amsmath}
\usepackage{amssymb}
\usepackage{amsthm}
\usepackage{pgfplots}
\usepackage{pgf}
\usepackage{tikz}
\usetikzlibrary{patterns}
\usepgfplotslibrary{patchplots} % LATEX and plain TEX
\usetikzlibrary{pgfplots.patchplots} % LATEX and plain TEX
\pgfplotsset{width=9cm,compat=1.5.1}

\usepackage[english]{babel}
\usepackage{amsthm}
\newtheorem{theorem}{Theorem}

\newtheorem{definition}[theorem]{Definition}
\newtheorem{lemma}[theorem]{Lemma}
\newtheorem{corollary}[theorem]{Corollary}
\newtheorem{remark}[theorem]{Remark}
\newtheorem{proposition}[theorem]{Proposition}

\newtheorem{example}[theorem]{Example}

\newcommand{\up}[1]{\overset{#1}{\uplus}~}

\usepackage{xcolor}
\usepackage{makeidx}
\usepackage[colorlinks=true,linkcolor=blue,anchorcolor=blue,citecolor=red,urlcolor=magenta]{hyperref}
\usepackage[alphabetic,backrefs]{amsrefs}

\begin{document}
\title{New results in vertex sedentariness}

\author{
	Hermie Monterde \textsuperscript{\!\!1}
}

\maketitle

%%%%%%%%%%%%%%%%%%%%%%%
%% Elsevier bibliography styles
%%%%%%%%%%%%%%%%%%%%%%%
%% To change the style, put a % in front of the second line of the current style and
%% remove the % from the second line of the style you would like to use.
%%%%%%%%%%%%%%%%%%%%%%%

%% Numbered
%\bibliographystyle{model1-num-names}

%% Numbered without titles
%\bibliographystyle{model1a-num-names}

%% Harvard
%\bibliographystyle{model2-names.bst}\biboptions{authoryear}

%% Vancouver numbered
%\usepackage{numcompress}\bibliographystyle{model3-num-names}

%% Vancouver name/year
%\usepackage{numcompress}\bibliographystyle{model4-names}\biboptions{authoryear}

%% APA style
%\bibliographystyle{model5-names}\biboptions{authoryear}

%% AMA style
%\usepackage{numcompress}\bibliographystyle{model6-num-names}

%% `Elsevier LaTeX' style
%\bibliographystyle{elsarticle-num}
%%%%%%%%%%%%%%%%%%%%%%%

\begin{abstract}
A vertex in a graph is said to be sedentary if a quantum state assigned on that vertex tends to stay on that vertex. Under mild conditions, we show that the direct product and join operations preserve vertex sedentariness. We also completely characterize sedentariness in blow-up graphs. These results allow us to construct new infinite families of graphs with sedentary vertices. We prove that a vertex with a twin is either sedentary or admits pretty good state transfer. Moreover, we give a complete characterization of twin vertices that are sedentary, and provide sharp bounds on their sedentariness. As an application, we determine the conditions in which perfect state transfer, pretty good state transfer and sedentariness occur in complete bipartite graphs and threshold graphs of any order.
\end{abstract}

	

%05C50, 15A18, 05C22, 81P45, 81A10

\addtocounter{footnote}{1}
\footnotetext{Department of Mathematics, University of Manitoba, Winnipeg, MB, Canada R3T 2N2}

%\tableofcontents
%\linenumbers

\section{Introduction}\label{secINTRO}

The concept of sedentariness was introduced by Godsil to study the `stay-at-home' behaviour of vertices in families of graphs as their number of vertices tend to infinity \cite{SedQW}. In particular, Godsil showed that complete graphs and large classes of strongly regular graphs are sedentary families. Recently, Monterde formalized the notion of a sedentary vertex and systematically studied both sedentary families and sedentary vertices \cite{Monterde2023}. She proved that Cartesian products preserve vertex sedentariness, and provided examples of strongly cospectral vertices that are sedentary. She also showed that a vertex with at least two twins is sedentary and established lower bounds on their sedentariness. Her results allowed for the construction of new graphs with sedentary vertices using joins, blow-ups and attachment of pendent paths.

In this paper, we continue the investigation of vertex sedentariness. In Section \ref{secPRELIM}, we introduce some graph theory, matrix theory and quantum walks background. Section \ref{secSed} is devoted to a summary of known results about sedentariness that are useful throughout the paper. We highlight the fact that pretty good state transfer and sedentariness are mutually exclusive types of state transfer. We also provide infinite families of graphs containing vertices that are neither sedentary nor involved in pretty good state transfer (Example \ref{star} and Remark \ref{exsc}). In Section \ref{secGT}, we prove an interesting result which states that a vertex with a twin only exhibits either sedentariness or pretty good state transfer (Corollary \ref{yipee1}). This allows us to characterize twin vertices that are sedentary (Theorem \ref{twinschar}) and provide an improved bound on their sedentariness (Theorem \ref{sed2}). We then apply our results to characterize pretty good state transfer and sedentariness in complete multipartite graphs in Section  \ref{secCMG} and threshold graphs in Section \ref{secTG}. In particular, we show that a huge fraction of complete multipartite graphs are Laplacian sedentary at every vertex (Corollary \ref{cmg1}). In fact, the only complete multipartite graphs that do not contain a Laplacian sedentary vertex are cocktail party graphs on twice an even number of vertices (Corollary \ref{23}). For threshold graphs, we characterize Laplacian sedentary vertices (Theorem \ref{psed}), and provide tight lower bounds on sedentariness, which vary depending on the cell containing the vertex and the form of the threshold graph in question (Theorem \ref{thresh}). We also establish that almost all complete multipartite graphs and threshold graphs contain a sedentary vertex (Corollary \ref{almost}). Section \ref{secDP} explores sedentariness in direct products. We give sufficient conditions such that the direct product of a complete graph with another graph yields sedentary vertices (Theorem \ref{dirprod}(1)). This yields new families of graphs with sedentary vertices, such as the direct product of complete graphs under mild conditions (Theorem \ref{complete}). We also characterize sedentariness in bipartite doubles (Theorem \ref{dirprod}(2)). In particular, we show that the bipartite double of a complete graph does not exhibit sedentariness (Corollary \ref{kmxY1}(1)). This prompts us to provide an infinite family of connected bipartite doubles containing sedentary vertices (Theorem \ref{k2xYsed}). Section \ref{secBU} is dedicated to characterizing adjacency sedentariness in blow-up graphs. Prior to this work, Monterde showed that the blow-ups of $m\geq 3$ copies of a graph are sedentary at every vertex. Here, we resolve the case $m=2$ (Corollary \ref{m=2}) and provide sharper bounds for sedentariness in blow-ups (Theorem \ref{size3}). In Section \ref{secJoins}, we prove that the join operation preserves  adjacency and Laplacian sedentariness under mild conditions (Theorem \ref{joinsed}). This provides yet another method for constructing new graphs with sedentary vertices. Finally, we present open questions in Section \ref{secOQ}.
%3e4r4te5

\section{Preliminaries}\label{secPRELIM}

Throughout, we let $X$ be a weighted and undirected graph with possible loops, no multiple edges, all positive edge weights and vertex set $V(X)$. The \textit{adjacency matrix} $A(X)$ of $X$ is defined entrywise as
\begin{equation*}
A(X)_{u,v}=
\begin{cases}
 \omega_{u,v}, &\text{if $u$ and $v$ are adjacent}\\
 0, &\text{otherwise},
\end{cases}
\end{equation*}
where $\omega_{u,v}$ is the weight of the edge $(u,v)$. The \textit{degree matrix} $D(X)$ of $X$ is the diagonal matrix of vertex degrees of $X$, where $\operatorname{deg}(u)=2\omega_{u,u}+\sum_{j\neq u}\omega_{u,j}$ for each $u\in V(X)$. For each $q\in\mathbb{R}$, we call the matrix $M_q(X)=qD(X)+A(X)$ a \textit{generalized adjacency matrix} for $X$. Note that $M_0(X)=A(X)$ and $M_{-1}(X)=-L(X)$, where $L(X)$ is the \textit{Laplacian matrix} of $X$. We say that $X$ is \textit{weighted $k$-regular} if $\operatorname{deg}(u)-\omega_{u,u}$ is a constant $k$ for each $u\in V(X)$. Equivalently, each row sum of $A(X)$ is a constant $k$. If $X$ is simple (i.e., loopless), then $X$ is weighted $k$-regular if and only if $D(X)=kI$. If the context is clear, then we write $M_q(X)$, $A(X)$, and $L(X)$ resp.\ as $M_q$, $A$, and $L$. We also write $M_q$ as $M$ when $q$ is not needed. We represent the transpose of $M$ by $M^T$, and the characteristic polynomial of $M$ in the variable $t$ by $\phi(M,t)$. We denote the simple unweighted empty, complete, cycle and path graphs on $n$ vertices resp.\ by $O_n$, $K_n$, $C_n$ and $P_n$. The all-ones vector of order $n$, the zero vector of order $n$, the $m\times n$ all-ones matrix, and the $n\times n$ identity matrix are denoted by $\textbf{1}_n$, $\textbf{0}_n$, $\textbf{J}_{m,n}$ and $I_n$, respectively. If $m=n$, then we write $\textbf{J}_{m,n}$ as $\textbf{J}_n$. If the context is clear, then we simply write these matrices resp.\ as $\textbf{1}$, $\textbf{0}$, $\textbf{J}$ and $I$.
%We say that $X$ is \textit{weighted $k$-regular} if every row sum of $A(X)$ equals $k$. Note that for a weighted $k$-regular graph $X$, $D(X)=kI$ whenever $X$ is simple.

Let $H$ be a Hermitian matrix whose rows and columns are indexed by the vertices of $X$ satisfying the property $H_{u,v}=0$ if and only if no edge exists between $u$ and $v$ in $X$. A \textit{(continuous-time) quantum walk} on $X$ with respect to $H$ is determined by the unitary matrix
\begin{equation}
\label{M}
U_H(t)=e^{itH}.
\end{equation}
In this work, we focus on the case $H=M_q$. If $q$ is not important, then we simply write $U_{M_q}(t)$ as $U_M(t)$. Unless otherwise specified, the results in this paper apply to $M_q=M$ for all $q\in\mathbb{R}$ Since $M$ is real symmetric, it admits a spectral decomposition

\begin{equation}
\label{specdec}
M=\sum_{j}\lambda_jE_j,
\end{equation}
where the $\lambda_j$'s are the distinct real eigenvalues of $M$ and each $E_j$ is the orthogonal projection matrix onto the eigenspace associated with $\lambda_j$. The \textit{eigenvalue support} of $u\in V(X)$ is the set
\begin{center}
    $\sigma_u(M)=\{\lambda_j:E_{j}\textbf{e}_u\neq \textbf{0}\},$
\end{center}
where $\textbf{e}_u$ denotes the characteristic vector of vertex $u$. From (\ref{specdec}), we may write (\ref{M}) as
\begin{center}
$U_M(t)=\sum_{j}e^{it\lambda_j}E_j$.
\end{center}
We say that two vertices $u$ and $v$ are \textit{cospectral} if $(E_j)_{u,u}=(E_j)_{v,v}$ for each $j$, and \textit{strongly cospectral} if $E_j\textbf{e}_u=\pm E_j\textbf{e}_v$ for each $j$. Note that if $u$ and $v$ are cospectral, then $U_M(t)_{u,u}=U_M(t)_{v,v}$ for all $t$.

With respect to the matrix $M$, we say that \textit{perfect state transfer} (PST) occurs between two vertices $u$ and $v$ at time $\tau$ if $|U(\tau)_{u,v}|=1$. The minimum $\tau>0$ such that $|U(\tau)_{u,v}|=1$ is called the \textit{minimum PST time} between $u$ and $v$. We say that $u$ is \textit{periodic} at time $\tau$ if $|U(\tau)_{u,u}|=1$. The minimum $\tau>0$ such that $|U(\tau)_{u,u}|=1$ is called the \textit{minimum period} of $u$, which we denote by $\rho$. We say that \textit{pretty good state transfer} (PGST) occurs between $u$ and $v$ if $\lim_{\tau_k\rightarrow\infty}|U(\tau_k)_{u,v}|=1$ for some $\{\tau_k\}\subseteq\mathbb{R}$. Further, if $X$ is a simple weighted $k$-regular graph, then $U_M(t)=e^{it(qkI+A)}=e^{it(qk)}U_A(t)$, and so for all $q\in\mathbb{R}$,
\begin{center}
$\left|U_M(t)_{u,v}\right|^2=\left|U_A(t)_{u,v}\right|^2.$
\end{center}
holds for all $t$ and for all $u,v\in V(X)$. That is, the quantum walks with respect to $A$ and $M$ are equivalent for all $q\in\mathbb{R}$, i.e., they exhibit the same types of state transfer. 

The \textit{join} of weighted graphs $X$ and $Y$, denoted $X\vee Y$, is the resulting graph obtained by joining every vertex of $X$ with every vertex of $Y$ with an edge of weight one.
In particular, the graph $O_1\vee Y$ is called a \textit{cone} on $Y$ and the lone vertex of $O_1$ is called the \textit{apex}. If $X\in\{O_2,K_2\}$, then $X\vee Y$ is called a \textit{double cone} on $Y$ and the two vertices of $X$ are called the \textit{apexes}. We denote the cocktail party graph on $2k$ vertices by $CP(2k)$, and the complete graph on $n$ vertices minus an edge by $K_n\backslash e$.

\section{Sedentariness}\label{secSed}

In this section, we gather some known results on sedentariness that will be useful throughout the paper.

\begin{definition}
\label{def}
We say that vertex $u$ of $X$ is \textit{$C$-sedentary} if for some constant $0<C\leq 1$,
\begin{equation}
\label{inf}
\inf_{t>0}\lvert U_M(t)_{u,u}\rvert\geq C.
\end{equation}
If equality holds in (\ref{inf}), then we say that $u$ is \textit{sharply} $C$-sedentary, while if the infimum in (\ref{inf}) is attained at some $t_0>0$, then we say that $u$ is \textit{tightly} $C$-sedentary at time $t=t_0$. Further, we say that $u$ is \textit{not sedentary} if $\inf_{t>0}\lvert U_M(t)_{u,u}\rvert=0$.
\end{definition}
If $C$ is not important, then we respectively say \textit{sedentary}, \textit{sharply sedentary}, and \textit{tightly sedentary}. Sedentariness of a vertex $u$ implies that $\lvert U_M(t)_{u,u}\rvert$ is bounded away from $0$, and so physically speaking, the quantum state initially at vertex $u$ tends to stay at $u$.

We now state an important property of sedentary vertices \cite[Proposition 2c]{Monterde2023}.

\begin{proposition}
\label{sedpgst}
A sedentary vertex cannot be involved in pretty good state transfer. Conversely, a vertex involved in pretty good state transfer cannot be sedentary.
\end{proposition}

Proposition \ref{sedpgst} implies that sedentarines and PGST are mutually exclusive types of quantum state transfer, although it is possible that a given vertex in a graph exhibits neither sedentarines nor PGST. We illustrate the latter statement by providing an infinite family of graphs that satisfy this property.

\begin{example}
\label{star}
For each $n\geq 3$, consider  $K_{1,n}$ with central vertex $u$. Then $U_A(\frac{\pi}{2\sqrt{n}})_{u,u}=0$, and so $u$ is not sedentary. Moreover, since $u$ is not strongly cospectral with any vertex in $X$, it cannot exhibit PGST.
\end{example}

In Remark \ref{exsc}, we provide an infinite family of graphs containing periodic and strongly cospectral vertices that are neither sedentary nor involved in PST with respect to the adjacency and Laplacian matrix.

The following result is due to Monterde \cite[Theorem 11]{Monterde2023}.

\begin{theorem}
\label{eureka}
Let $u$ be a vertex of $X$ with $\sigma_u(M)=\{\lambda_1,\ldots,\lambda_r\}$, where $E_j$ is the orthogonal projection matrix corresponding to $\lambda_j$. If $S$ is a non-empty proper subset of $\sigma_u(M)$, say $S=\{\lambda_1,\ldots,\lambda_s\}$, such that
\begin{equation}
\label{eureka1}
\sum_{j=1}^s(E_j)_{u,u}=a
\end{equation}
for some $\frac{1}{2}\leq a<1$, then
\begin{equation}
\label{eureka2}
\lvert U_M(t)_{u,u}\rvert\geq \bigg|\sum_{j=1}^s e^{it\lambda_j}(E_j)_{u,u}\bigg|- (1-a)\quad \text{for all}\ t.
\end{equation}
Moreover, if there exists a time $t_1>0$ such that $\bigg|\sum_{j=1}^s e^{it_1\lambda_j}(E_j)_{u,u}\bigg|\geq 1-a$, and for all $j\in\{1,\ldots,s\}$ and $k\in\{s+1,\ldots,r\}$,
\begin{equation}
\label{eureka333}
e^{it_1(\lambda_1-\lambda_j)}=1\quad  \text{and}\quad e^{it_1(\lambda_{1}-\lambda_{k})}=-1,
\end{equation}
then equality holds in (\ref{eureka2}), in which case $\lvert U_M(t_1)_{u,u} \rvert=2a-1$ and $u$ is periodic at time $2t_1$.
\end{theorem}

\begin{remark}
From (\ref{eureka2}), it suffices to find a non-empty proper subset $S$ of $\sigma_u(M)$ with $\sum_{j\in S}(E_j)_{u,u}=a\geq \frac{1}{2}$ such that $ \left\lvert\sum_{j=1}^s e^{it\lambda_j}(E_j)_{u,u}\right\rvert- (1-a)$ is bounded away from 0 for all $t$. 
\end{remark}
%for some $\frac{1}{2}\leq a<1$

We also have the following result, which characterizes the occurrence of (\ref{eureka333}). In what follows, we denote the largest power of two that divides an integer $a$ by $\nu_2(a)$.

\begin{lemma}
\label{ts}
Let $u\in V(X)$ and suppose $S\subseteq \sigma_u(M)$ satisfies (\ref{eureka1}). If $\phi(M,t)\in\mathbb{Z}[t]$, then (\ref{eureka333}) holds if and only if all of the following conditions hold
\begin{enumerate}
\item Either (a) all eigenvalues in $\sigma_{u}(M)$ are integers or (b) all eigenvalues in $\lambda_j\in \sigma_{u}(M)$ are quadratic integers of the form $\frac{1}{2}(a+b_j\sqrt{\Delta})$, where $a,b_j,\Delta$ are integers and $\Delta>1$ is square-free.
\item For all $\lambda_j,\lambda_m\in S$ and $\lambda_k,\lambda_{\ell}\in\sigma_{u}(M)\backslash S$, 
    $\nu_2(\frac{\lambda_j-\lambda_m}{\sqrt{\Delta}})>\nu_2(\frac{\lambda_k-\lambda_m}{\sqrt{\Delta}})=\nu_2(\frac{\lambda_{\ell}-\lambda_m}{\sqrt{\Delta}})$, where $\Delta=1$ whenever (1a) holds.
\end{enumerate}
Moreover, the minimum time in which (\ref{eureka333}) holds is $t_1=\frac{\pi}{g}$, where $g= \operatorname{gcd}(\mathcal{T})$, $\mathcal{T}=\{\frac{\lambda_1-\lambda}{\sqrt{\Delta}}: \lambda\in\sigma_{u}(M)\}$ and $\lambda_1\in S$ is fixed, and $u$ is periodic with $\rho=2t_1$.
\end{lemma}

\begin{remark}
If $u$ and $v$ are strongly cospectral in $X$ and Lemma \ref{ts}(1-2) holds, then PST occurs between $u$ and $v$ at $t_1$ and $S=\sigma_{uv}^+(M)$. In this case, $a=\frac{1}{2}$ so (\ref{eureka1}) yields $U_M(t_1)_{u,u}=0$. Thus, $u$ is not sedentary. %This is not suprising as the conditions yield PST between $u$ and $v$ by \cite[Theorem 2.4.4]{Coutinho2014}.
\end{remark}

We also state Lemmas 9 and 12 in \cite{Monterde2023}, respectively.

\begin{lemma}
\label{persed}
Let $u$ be a periodic vertex in $X$. Then $u$ is sedentary if and only if $U_M(t)_{u,u}\neq 0$ for all $t\in[0,\rho]$. In particular, if $u$ is sedentary, then it is tightly $C$-sedentary, where $C=\min_{t\in[0,\rho]} |U(t)_{u,u}|$.
\end{lemma}

\begin{lemma}
\label{eurekarem2}
Let $u$ be a vertex of $X$ with $\sigma_u(M)=\{\lambda_1,\ldots,\lambda_r\}$, where $E_j$ is the orthogonal projection matrix corresponding to $\lambda_j$. Suppose $S$ is a non-empty proper subset of $\sigma_u(M)$, say $S=\{\lambda_1,\ldots,\lambda_s\}$, such that (\ref{eureka1}) holds for some $\frac{1}{2}\leq a<1$. The following are equivalent.
\begin{enumerate}
\item If $\ell_j$ and $m_j$ are integers such that $\sum_{j=1}^sm_j\lambda_j+\sum_{j=s+1}^r\ell_j\lambda_j=0$ and $\sum_{j=1}^sm_j+\sum_{j=s+1}^r\ell_j=0$, then $\sum_{j=1}^sm_j$ is even.
\item There exists a sequence $\{t_k\}\subseteq \mathbb{R}$ such that $\lim_{k\rightarrow\infty}\lvert U_M(t_k)_{u,u} \rvert=2a-1$.
\end{enumerate}
\end{lemma}

We close this section with a remark that Lemmas \ref{ts} and \ref{eurekarem2} are equivalent for periodic vertices.

\section{Graphs with twins}\label{secGT}

Denote the neighbourhood of a vertex $u$ in $X$ by $N_X(u)$. Two vertices $u$ and $v$ of $X$ are \textit{twins} if
\begin{enumerate}
\item $N_X(u)\backslash \{u,v\}=N_X(v)\backslash \{u,v\}$,
\item the edges $(u,w)$ and $(v,w)$ have the same weight for each $w\in N_X(u)\backslash \{u,v\}$, and
\item the loops on $u$ and $v$ have the same weight, and this weight is zero if those loops are absent.
\end{enumerate}
We allow twins to be adjacent. A maximal subset $T=T(\omega,\eta)$ of $V(X)$ with at least two elements is a \textit{twin set} in $X$ if the vertices in $T$ are pairwise twins, each $u\in T$ has a loop of weight $\omega$, and every pair of vertices in $T$ are connected by an edge with weight $\eta$. In particular, if $X$ is a simple unweighted graph, then $\omega=0$ and $\eta\in\{0,1\}$. As an example, the apexes of a double cone form a twin set of size two. Note that the vertices in a twin set $T$ are pairwise cospectral \cite[Corollary 2.11]{Monterde2022}. Thus, if $u,v\in T$, then $U_M(t)_{u,u}=U_M(t)_{v,v}$ for all $t$, and so a vertex in $T$ is sedentary if and only if each vertex in $T$ is sedentary.

The following fact is due to Monterde \cite[Lemma 2.9]{Monterde2022}.

\begin{lemma}
\label{alphabeta}
Let $T=T(\omega,\eta)$ be a twin set in $X$. Then $u,v\in T$ if and only if
$\textbf{e}_u-\textbf{e}_v$ is an eigenvector associated to the eigenvalue $\theta=q \operatorname{deg}(u)+\omega-\eta$ of $M$.%If we add that $X$ is unweighted, then $\theta=q \operatorname{deg}(u)+\omega-1$ whenever $u$ and $v$ are adjacent, and
\end{lemma}

Our next result is due to Kirkland et al.\ \cite[Corollary 1]{Kirkland2023}.

\begin{theorem}
\label{yipee}
Let $T$ be a twin set in $X$. Then for any two vertices $u$ and $v$ in $T$,
\begin{center}
$\lvert U_M(t)_{u,u}\rvert+ \lvert U_M(t)_{u,v}\rvert\geq 1\quad $ for all $t\in\mathbb{R}$.
\end{center}
\end{theorem}

\begin{corollary}
\label{yipee1}
Let $T$ be a twin set in $X$. Then a vertex $u\in T$ is either sedentary or involved in pretty good state transfer with some vertex $v\neq u$ in $X$. If the latter holds, then $T=\{u,v\}$. %Furthermore, if $u\in T$ is not involved in strong cospectrality, then $u\in T$ is sedentary.
\end{corollary}

\begin{proof}
Theorem \ref{yipee} yields the first statement. In particular, if $u\in T$ is  involved in PGST with $v$, then they are strongly cospectral, and so \cite[Theorem 3.9(2)]{Monterde2022} yields the second statement. %The last statement is immediate from the first.
\end{proof}
 
Since strong cospectrality is a necessary condition for PGST, Corollary \ref{yipee1} yields the next result.

\begin{corollary}
\label{yipee2}
Let $T$ be a twin set in $X$. If $u\in T$ is not involved in strong cospectrality, then each $u\in T$ is sedentary.
\end{corollary}

We now prove one of our main results, which characterizes sedentariness in twin vertices.

\begin{theorem}
\label{twinschar}
Let $T$ be a twin set in $X$ and consider the eigenvalue $\theta$ in Lemma \ref{alphabeta}. Then each vertex in $T$ is sedentary if and only if one of the conditions below hold.
\begin{enumerate}
\item Either (i) $|T|\geq 3$ or (ii) $T=\{u,v\}$ and there is an eigenvector $\textbf{w}\notin\operatorname{span}\{\textbf{e}_u-\textbf{e}_v\}$ of $M$ associated with $\theta$ such that $\textbf{w}^T\textbf{e}_u\neq 0$ or $\textbf{w}^T\textbf{e}_v\neq 0$. Here, any vertex in $T$ cannot exhibit strong cospectrality.%, then each vertex in $T$ is sedentary.
\item $\theta$ is a simple eigenvalue of $M$, so that $T=\{u,v\}$ and $u$ and $v$ are strongly cospectral vertices, and there are integers $m_j$ such that
\begin{equation*}
\sum_{\lambda_j\in\sigma_{uv}^+(M)}m_j(\lambda_j-\theta)=0\quad \text{and}\quad \sum_{\lambda_j\in\sigma_{uv}^+(M)}m_j\ \text{is odd}.
\end{equation*}
If we add that $\phi(M,t)\in\mathbb{Z}[t]$ and $u$ is periodic, then the latter statement is equivalent to each eigenvalue $\lambda_j\in\sigma_{uv}^+(M)$ is of the form $\lambda_j=\theta+b_j\sqrt{\Delta}$, where $b_j$ is an integer and either $\Delta=1$ or $\Delta>1$ is a square-free integer and the $\nu_2(b_j)$'s are not all equal. In this case, $u$ is tightly sedentary.
\end{enumerate}
\end{theorem}
\begin{proof}
Combining condition (1) with \cite[Corollary 3.14]{Monterde2022} implies that each vertex in $T$ cannot be strongly cospectral with another vertex. Thus, each vertex in $T$ is sedentary by Corollary \ref{yipee2}. This proves (1). To prove (2), note that the condition that $\theta$ is a simple eigenvalue of $M$ implies that $T=\{u,v\}$ and $u$ and $v$ are strongly cospectral \cite[Corollary 3.14]{Monterde2022}. Combining this with Corollary \ref{yipee1} and the characterizations of PST and PGST between twin vertices (see Theorems 11 and 14 in \cite{Kirkland2023}) yields the desired results in (2). The last statement in (2) follows from Lemma \ref{persed}.
\end{proof}

We now prove another main result in this paper, which is an improvement of \cite[Theorem 16]{Monterde2023}.

\begin{theorem}
\label{sed2}
Let $T$ be a twin set in $X$ and fix $u\in T$. Let $\mathcal{B}_1$ be the resulting set after orthonormalizing $\{\textbf{e}_u-\textbf{e}_v: v\in T\backslash\{u\}\}$, and let $ \mathcal{B}= \mathcal{B}_1\cup \mathcal{B}_2$ be an orthonormal basis for the eigenspace of $M$ associated with the eigenvalue $\theta$ in Lemma \ref{alphabeta}. Define $F=\sum_{ \textbf{w}\in \mathcal{B}_2} \textbf{w}\textbf{w}^T$, where $F$ is absent whenever $\mathcal{B}_2=\varnothing$. Then
\begin{equation}
\label{sedeq}
|U_M(t)_{u,u}|\geq 2(E_{\theta})_{u,u}-1= 1-\frac{2}{|T|}+2F_{u,u}\quad \text{for all $t$},
\end{equation}
with equality if and only if (\ref{eureka333}) holds. The following also hold.
\begin{enumerate}
\item Suppose $1-\frac{2}{|T|}+2F_{u,u}>0$.
\begin{enumerate}
    \item If $\phi(M,t)\in\mathbb{Z}[t]$ and each vertex in $T$ is periodic, then each vertex in $T$ is tightly $C$-sedentary, where $C=1-\frac{2}{|T|}+2F_{u,u}$ whenever Lemma \ref{ts} holds and $C>1-\frac{2}{|T|}+2F_{u,u}$ otherwise. 
\item Suppose each vertex in $T$ is not periodic. Then each vertex in $T$ is sharply $(1-\frac{2}{|T|}+2F_{u,u})$-sedentary if and only if Lemma \ref{eurekarem2} holds.
\end{enumerate}
\item If $|T|\geq 3$, then each vertex in $T$ is $\left(1-\frac{2}{|T|}+2F_{u,u}\right)$-sedentary.
\item If $T=\{u,v\}$ and there is an eigenvector $\textbf{w}\notin\operatorname{span}\{\textbf{e}_u-\textbf{e}_v\}$ of $M$ associated with $\theta$ such that $\textbf{w}^T\textbf{e}_u\neq 0$ or $\textbf{w}^T\textbf{e}_v\neq 0$, then each vertex $T$ is $2F_{u,u}$-sedentary.
\end{enumerate}
\end{theorem}

\begin{proof}
Let $0_m$ denote the $m\times m$ zero matrix. From our assumption, we deduce that
\begin{equation*}
E_{\theta}=\left(I_{|T|}-\frac{1}{|T|}\textbf{J}_{|T|}\oplus 0_{n-|T|}\right)+F,
\end{equation*}
Taking $S=\{\theta\}$ and $a=(E_\theta)_{u,u}=1-\frac{1}{|T|}+F_{u,u}\geq \frac{1}{2}$ in Theorem \ref{eureka} yields
%Since the $E_j$'s sum to identity, $\displaystyle\sum_{j=2}^r(E_j)_{u,u}=\frac{1}{|T|}-F_{u,u}>0$. Thus, for all $t$,
\begin{equation}
\label{abeq}
\begin{split}
|U_M(t)_{u,u}|&\geq a-(1-a)=2a-1=1-\frac{2}{|T|}+2F_{u,u}.
\end{split}
\end{equation}
Note that (1) is immediate from Lemma \ref{ts} and Lemma \ref{eurekarem2}, while (2) follows directly from (\ref{abeq}). Moreover, the assumption in (4) implies that $F_{u,u}>0$, and so (\ref{abeq}) yields $|U_M(t)_{u,u}|\geq 2F_{u,u}>0$. 
\end{proof}

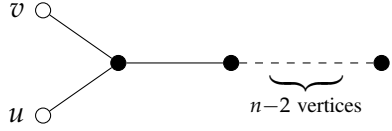
\begin{figure}[h!]\label{fig:small} 
	\begin{center}
		\begin{tikzpicture}
		\tikzset{enclosed/.style={draw, circle, inner sep=0pt, minimum size=.2cm}}	   
	   \node[enclosed, label={left, yshift=0cm: $u$}] (w_1) at (0.5,1.8) {};
	    \node[enclosed, label={left, yshift=0cm: $v$}] (w_2) at (0.5,3.2) {};
		\node[enclosed, fill=black] (w_3) at (1.5,2.5) {};
		\node[enclosed, fill=black] (w_4) at (3,2.5) {};
		\node[enclosed, fill=black] (w_6) at (5,2.5) {};
		\draw (w_1) -- (w_3);
		\draw (w_2) -- (w_3);
		\draw (w_4) -- (w_3);
		\draw (w_4) -- node[below] {$\underbrace{}_{n-2\ \text{vertices}}$} (w_6)[dashed];
		\end{tikzpicture}
	\end{center}
	\caption{The graph $P_n'$ with twin vertices $u$ and $v$ marked white}\label{yay}
\end{figure}

\begin{remark}
\label{remsed}
Let $T$ be a twin set in $X$. If Theorem \ref{twinschar}(1) holds, then Theorem \ref{sed2}(2-3) provides a bound for the sedentariness of vertices in $T$. Now, if Theorem \ref{twinschar}(2) holds, then $\mathcal{B}=\mathcal{B}_1=\operatorname{span}\{\textbf{e}_u-\textbf{e}_v\}$ in Theorem \ref{sed2}, and so (\ref{sedeq}) yields the trivial inequality $|U_M(t)_{u,u}|\geq 0$. In this case, $T=\{u,v\}$ and $u$ and $v$ are strongly cospectral in $X$, and so one may invoke Theorem \ref{twinschar}(2) to determine whether $u$ is sedentary in $X$ or not. By explicitly calculating lower bounds for $|U_M(t)|_{u,u}$, Monterde provided sharp bounds for the sedentariness of the apexes of $O_2\vee Y$ for $M\in\{A,L\}$ with the additional condition that $Y$ is regular when $M=A$ (see Theorems 32 and 36 in \cite{Monterde2023}). But in general, if $u\in T$ is sedentary in $X$, then we currently do not have a straightforward way of determining a bound for the sedentariness of vertices in $T$.
%But since this approach is tedious, a simple and a more general technique would be.
\end{remark}

Monterde first established Theorem \ref{sed2}(2) in \cite[Theorem 16]{Monterde2023}, and she used this result provide infinite families of join graphs, blow-up graphs and graphs with tails that contain sedentary vertices. In the next examples, we provide infinite families of graphs that satisfy Theorem \ref{sed2}(3).

\begin{example}
\label{exsed}
Let $M=L$ and $Y$ be a simple weighted graph on $n$ vertices. Consider $K_2\vee Y$ and let $u=1$ and $v=2$ be the apexes of $K_2\vee Y$. Then $\sigma_u(L(K_2\vee Y))=\{0,\theta\}$, where $\theta=n+2$. If $\mathcal{B}$ is an orthonormal basis for the eigenspace associated with the eigenvalue $n$ of $L(Y)$, then
\begin{center}
    $\left\{\frac{1}{\sqrt{2}}(\textbf{e}_1-\textbf{e}_2), \frac{1}{\sqrt{2n(n+2)}}\left[ \begin{array}{ccccc} n\textbf{1}_2 \\ -2\textbf{1}_n\end{array} \right]\right\}\cup \left\{\left[ \begin{array}{ccccc} \textbf{0} \\ \textbf{v}\end{array} \right]:\textbf{v}\in\mathcal{B}\right\}$
\end{center}
is as an orthonormal basis for the eigenspace associated with the eigenvalue $\theta$ of $L(K_2\vee Y)$ and the set on the right is absent if $Y$ is not a join graph. Hence, if we let $\textbf{w}=\frac{1}{\sqrt{2n(n+2)}}\left[ \begin{array}{ccccc} n\textbf{1}_2 \\ -2\textbf{1}_n\end{array} \right]$, then the conditions in Theorem \ref{sed2}(3) are met, and we conclude that
\begin{center}
    $|U_{L(K_2\vee Y)}(t)_{u,u}|\geq 2F_{1,1}=\frac{n}{n+2}\quad \text{for all $t$}$
\end{center}
with equality if and only if $t=\frac{j\pi}{n+2}$ for all odd $j$. Thus, vertices $u$ and $v$ in $K_2\vee Y$ are tightly $(\frac{n}{n+2})$-sedentary. This result is consistent with \cite[Theorem 29]{Monterde2023}.
\end{example}

\begin{example}
\label{exsed}
Let $M=A$, $n\geq 3$ be odd and $u=1$ be an end vertex of $P_n$. Let $P_n'$ be the resulting graph after adding vertex $v:=n+1$ to $P_n$ such that $u$ and $v$ non-adjacent twins (see Figure \ref{fig:small}). If $\textbf{x}=\frac{1}{\sqrt{2}}(\textbf{e}_1-\textbf{e}_{n+1})$ and $\textbf{w}=\frac{\sqrt{2}}{\sqrt{n+1}}[1,0,-1,0,1,\ldots,(-1)^{\frac{n-1}{2}},0]^T$, then $\textbf{x}$ and $\textbf{w}$ are eigenvectors asscociated with the eigenvalue $\theta=0$ of $A$. Orthonormalizing $\{\textbf{x},\textbf{w}\}$ yields an orthonormal basis $\{\textbf{x},\textbf{w}'\}$ for the eigenspace associated to $\theta$, where $\textbf{w}'=\frac{1}{\sqrt{2n}}[1,0,2,0,-2,\ldots,0,2^{\frac{n+1}{2}},1]^T$. Hence, Theorem \ref{sed2}(3) yields
\begin{center}
    $|U_{A(P_n')}(t)_{u,u}|\geq 2F_{1,1}=\frac{1}{n}\quad \text{for all $t$}.$
\end{center}
%Thus, $u$ and $v$ are $(\frac{1}{n})$-sedentary in $P_n'$. Moreover, i
If the nonzero eigenvalues in $\sigma_u(A(P_5'))$ are linearly independent over $\mathbb{Q}$, then a direct application of Lemma \ref{eurekarem2} implies that $u$ is sharply $(\frac{1}{n})$-sedentary in $P_n'$. In particular, if $n=3$, then $u$ is tightly $(\frac{1}{3})$-sedentary in $P_3'=K_{1,3}$, while if $n=5$, then $\sigma_u(A(P_5'))=\{0,\pm\sqrt{(5\pm\sqrt{5})/2}\}$, and so $u$ is sharply $(\frac{1}{5})$-sedentary in $P_5'$. We conjecture that $u$ is sharply $(\frac{1}{n})$-sedentary in $P_n'$ for all odd $n\geq 5$.
\end{example}

\section{Complete multipartite graphs}\label{secCMG}

A complete bipartite graph $K_{n_1,\ldots,n_k}$ is the graph $K_{n_1,\ldots,n_k}=\bigvee_{j=1}^k O_{n_j}$. If $n_{\ell}=1$ (resp., $n_{\ell}=2$) for some $\ell\in\{1,\ldots,k\}$, then we may write $K_{n_1,\ldots,n_k}=O_1\vee \bigvee_{j\neq \ell} O_{n_j}$ (resp., $K_{n_1,\ldots,n_k}=O_2\vee \bigvee_{j\neq \ell} O_{n_j}$), and so we may view $K_{n_1,\ldots,n_k}$ as a cone (resp., double cone) on $Y=\bigvee_{j\neq \ell} O_{n_j}$. Our goal in this section to characterize PGST and sedentariness in complete multipartite graphs with respect to $M\in\{A,L\}$.

\subsection*{Laplacian case}

We first analyze the quantum walk on a complete multipartite graphs using its Laplacian matrix.

\begin{theorem}
\label{cmg}
Let $X=K_{n_1,\ldots,n_k}$ and $n=\sum_{j=1}^kn_j$. If $\ell\in\{1,\ldots,k\}$ and $u$ is a vertex of $X$ that belong to the partite set of size $n_{\ell}$, then the following hold with respect to $L=L(X)$.
\begin{enumerate}
\item If $n_{\ell}=1$, then $u$ is tightly $(1-\frac{2}{n})$-sedentary at time $t=\frac{j\pi}{n}$ for any odd $j$.
\item If $n_{\ell}=2$ and $n\equiv 0$ (mod 4), then the two vertices of $X$ that belong to the partite set of size $n_{\ell}$ admit perfect state transfer with minimum PST time $\frac{\pi}{2}$.
\item If $n_{\ell}=2$ and $n\equiv 2$ (mod 4), then $u$ is tightly $(\frac{2}{n})$-sedentary at time $t=\frac{j\pi}{2}$ for any odd $j$.
\item Let $n_{\ell}=2$ and $n$ be odd. If $n=3$, then  $u$ is tightly $(\frac{1}{3})$-sedentary at time $t=j\pi$ for any odd $j$. If $n\geq 5$, then  $u$ is tightly $(\frac{\sqrt{2}}{n})$-sedentary at time $t=\frac{j\pi}{2}$ for any odd $j$.
\item If $n_{\ell}\geq 3$, then $u$ is tightly $C$-sedentary, where $C=1-\frac{2}{n_{\ell}}$ at time $t=\frac{j\pi}{g}$ whenever $\nu_2(n)>\nu_2(n_{\ell})$, where $g=\operatorname{gcd}(n,n_{\ell})$ and $j$ is any odd integer and $C>1-\frac{2}{n_{\ell}}$ otherwise.
\end{enumerate}
\end{theorem}

\begin{proof}
Since $L(X)$ has integer eigenvalues, (1), (2), (3-4) and (5) resp.\ follow from \cite[Theorem 29]{Monterde2023}, \cite[Corollary 4]{Alvir2016}, \cite[Theorem 32]{Monterde2023} and Theorem \ref{sed2}(2)-\cite[Theorem 35]{Monterde2023} combined.
\end{proof}

Since $CP(2k)=\bigvee_k O_2$, Theorem \ref{cmg}(2-3) yields our next result.

\begin{corollary}
\label{cmg11}
Let $X=CP(2k)$. If $k$ is even, then perfect state transfer occurs between pairs of non-adjacent vertices in $X$ with minimum PST time $\frac{\pi}{2}$. Otherwise, every vertex in $X$ is tightly $(\frac{1}{k})$-sedentary.
\end{corollary}

Next, since $K_n\backslash e=O_2\vee K_{n-2}$, Theorem \ref{cmg} gives us our next result.

\begin{corollary}
\label{cmg11kne}
For all $n\geq 3$, let $X=K_n\backslash e$ with non-adjacent vertices $u$ and $v$.
\begin{enumerate}
\item For all $n\geq 3$, each degree $n-1$ vertex of $X$ is $(1-\frac{2}{n})$-sedentary at time $t=\frac{j\pi}{n}$ for any odd $j$.
\item If $n\equiv 0$ (mod 4), then perfect state transfer occurs between $u$ and $v$ with minimum time $\frac{\pi}{2}$. Otherwise, $u$ is tightly $C$-sedentary, where $C=\frac{2}{n}$ whenever $n\equiv 2$ (mod 4) at time $t=\frac{j\pi}{2}$, $C=\frac{1}3$ whenever $n=3$ at time $t=j\pi$ and $C=\frac{\sqrt{2}}{n}$ whenever $n\geq 5$ is odd at time $t=\frac{j\pi}{2}$, where $j$ is odd.
\end{enumerate}
\end{corollary}
In contrast to complete graphs where each vertex is sedentary, we now show that a huge fraction of the family of complete multipartite bipartite graphs are Laplacian sedentary at every vertex.

\begin{corollary}
\label{cmg1}
Let $X=K_{n_1,\ldots,n_k}$ and $n=\sum_{j=1}^kn_j$. If either (i) $n\not\equiv 0$ (mod 4) or (ii) $n\equiv 0$ (mod 4) and $n_j\neq 2$ for each $j\in\{1,\ldots,k\}$, then every vertex in $X$ is sedentary.
\end{corollary}

If $n\equiv 0$ (mod 4), then $n$ only has one partition where all parts are equal to two. Thus:%, Theorem \ref{cmg} yields the following result.

\begin{corollary}
\label{23}
$CP(2k)$ for even $k$ are the only complete multipartite graphs with no sedentary vertex.
\end{corollary}

\subsection*{Adjacency case}

The next result is an analogue of Theorem \ref{cmg} for the adjacency case.

\begin{theorem}
\label{cmgA}
Let $X=K_{n_1,\ldots,n_k}$ and $n=\sum_{j=1}^kn_j$. If $\ell\in\{1,\ldots,k\}$ and $u$ is a vertex of $X$ that belong to the partite set of size $n_{\ell}$, then the following hold with respect to $A=A(X)$.
\begin{enumerate}
\item Let $n_{\ell}=1$ for each $\ell\in \mathcal{I}$, where $\varnothing\neq \mathcal{I}\subseteq\{1,\ldots,k\}$.
\begin{enumerate} 
\item If $|\mathcal{I}|=1$ and $n_r=m$ for all $r\notin \mathcal{I}$, then $u$ is tightly $(\frac{(n-m-1)^2}{(n-m-1)^2+4(n-1)})$-sedentary %$|U_A(t)_{u,u}|\geq \frac{(n-m-1)^2}{(n-m-1)^2+4(n-1)}$ with equality if and only if 
at time $t=\frac{j\pi}{\sqrt{(n-m-1)^2+4(n-1)}}$ for any odd integer $j$. 
\item Let $|\mathcal{I}|=2$ and $n_r=m$ for all $r\notin \mathcal{I}$. Let $\Delta=(n-m-3)^2+8(n-2)$.
\begin{enumerate}
    \item If $\Delta$ is not a perfect square, then pretty good state transfer occurs between the two vertices that belong to partite sets of size one.
    \item If $\Delta$ is a perfect square and $\nu_2(n-m+1)\neq \nu_2(\sqrt{\Delta})$, then perfect state transfer occurs between the two vertices that belong to partite sets of size one.
    \item If $\Delta$ is a perfect square and $\nu_2(n-m+1)=\nu_2(\sqrt{\Delta})$, %(i.e., $n-2=\frac{1}{2}s(n-m-3+s)$ for some integer $s$ such that $\nu_2(n-m+1)\leq \nu_2(s-2)$), 
    then $X$ is tightly sedentary at $u$. 
\end{enumerate}
    \item Let $3\leq |\mathcal{I}| \leq n$. Then $|U_{A}(t)_{u,u}|\geq 1- \frac{2}{|\mathcal{I}|}$ for all $t$. If we add that $n_r=m$ for all $r\notin\mathcal{I}$ and let $\Delta=(n-m-2|\mathcal{I}|+1)^2+4|\mathcal{I}|(n-|\mathcal{I}|)$, then the following hold.
    \begin{enumerate}
        \item Let $\Delta$ be a perfect square. %(i.e., $|\mathcal{I}|(n-|\mathcal{I}|)=(n-m-2|\mathcal{I}|+1)s+s^2$ for some integer $s$).
        If $\nu_2(n-m+1)\neq \nu_2(\sqrt{\Delta})$, %(i.e., $\nu_2(n-m+1)>\nu_2(s-|\mathcal{I}|)$), 
        then $u$ is tightly $(1-\frac{2}{{|\mathcal{I}|}})$-sedentary at time $t=\frac{j\pi}{g}$, where $g=\operatorname{gcd}(n-m+1+\sqrt{\Delta},n-m+1-\sqrt{\Delta})$ and $j$ is odd. %In particular, if $n=|\mathcal{I}|$, then $t=\frac{j\pi}{n}$.
        Otherwise, $u$ is tightly $C$-sedentary for some $C>1-\frac{2}{{|\mathcal{I}|}}$.
        \item If $\Delta$ is not a perfect square, then $u$ is sharply $(1-\frac{2}{{|\mathcal{I}|}})$-sedentary.%, where $C=1-\frac{2}{{|\mathcal{I}|}}$ if $n$ is odd and $C>1-\frac{2}{{|\mathcal{I}|}}$ otherwise.
        \end{enumerate}
    \item Let $|\mathcal{I}|\leq n-4$ and $n_r=2$ for all $r\notin\mathcal{I}$ (so $n-|\mathcal{I}|$ is even). Let $\Delta=(n-1)^2+4|\mathcal{I}|$.
    \begin{enumerate}
    \item If $\Delta$ is not a perfect square, then pretty good state transfer occurs between every pair of vertices in a partite set of size two if and only if $n\equiv 3$ (mod 4).
    \item If $\Delta$ is a perfect square, %(i.e., $|\mathcal{I}|=(n-1)s+s^2$ for some integer $s$ that divides $|\mathcal{I}|$), 
    then then perfect state transfer occurs between every pair of vertices in a partite set of size two if and only if $\nu_2(n-3)\neq \nu_2(\sqrt{\Delta})$. %(i.e., $\nu_2(n-3)>\nu_2(s+1)$).
    \item If either (i) $\Delta$ is not a perfect square and $n\not\equiv 3$ (mod 4), or (ii) $\Delta$ is a perfect square and $\nu_2(n-3)=\nu_2(\sqrt{\Delta})$, then each vertex in a partite set of size two is sedentary.   
\end{enumerate}
\end{enumerate}
\item Let $n_{\ell}=2$ and $n_r=m$ for all $r\neq \ell$. Let $\Delta=(n-m-2)^2+8(n-2)$ and $v$ be a twin of $u$.
\begin{enumerate}
    \item If either (i) $\Delta$ is not a perfect square, (ii) $\Delta$ is a perfect square and $\nu_2(n-m-2)\neq \nu_2(\sqrt{\Delta})$ or (iii) $n=m+2$, then $X$ is not sedentary at $u$. In particular, if (i) holds, then pretty good state transfer occurs between $u$ and $v$. Otherwise, perfect state transfer occurs between them.
    %$n=\frac{(m+2)s-s^2-4}{s-2}$
    \item Let $n>m+2$ and $n-2=\frac{1}{2}s(n-m-2+s)$ for some integer $s$ such that $\nu_2(n-m-2)\leq \nu_2(s)$. %(i.e., $\Delta$ is a perfect square and $\nu_2(n-m-2)=\nu_2(\sqrt{\Delta})$). 
    Let $d_1=(n-m-2)/g$ and $s_1=s/g$, where $g=\operatorname{gcd}(n-m-2,s)$. 
    \begin{enumerate}
        \item If $s_1=1$, then $|U_{A}(t)_{u,u}|\geq \frac{1}{d_1+2}$ for all $t$ with equality if and only if $t=\frac{j\pi}{g}$ for odd $j$.
        \item If $s_1\geq 2$, then $|U_{A}(t)_{u,u}|\geq \frac{\sqrt{2}}{d_1+2s_1}$ for all $t$ with equality if and only if $t=\frac{j\pi}{g}$ for odd $j$.
    \end{enumerate}
    %$\nu_2(d+\sqrt{\Delta})=\nu_2(d-\sqrt{\Delta})$
\end{enumerate}
\item If $n_{\ell}\geq 3$, then $|U_{A}(t)_{u,u}|\geq 1-\frac{2}{n_{\ell}}$ for all $t$. Moreover, $u$ is tightly $(1-\frac{2}{n_{\ell}})$-sedentary at time $t=\frac{j\pi}{g}$, whenever $n_r=m$ for all $r\neq \ell$, $\Delta=(n-n_{\ell}-m)^2+4n_{\ell}(n-n_{\ell})$ is a perfect square and $\nu_2(n-n_{\ell}-m)\neq \nu_2(\sqrt{\Delta})$, %(i.e., $n_{\ell}(n-n_{\ell})=(n-n_{\ell}-m)s+s^2$ for some integer $s$ such that $\nu_2(n-n_{\ell}-m)\geq \nu_2(s)+1$), 
where $g=\operatorname{gcd}(n-n_{\ell}-m+\sqrt{\Delta},n-n_{\ell}-m-\sqrt{\Delta})$ and $j$ is odd. Furthermore, if $\sigma_u(A(X))$ is a linearly independent set over $\mathbb{Q}$, then $u$ is sharply $(1-\frac{2}{n_{\ell}})$-sedentary.
\end{enumerate}
\end{theorem}

\begin{proof}
    In (1a), $X$ can be viewed as a cone over an $(n-m-1)$-regular graph on $n-1$ vertices with apex $u$. If $X$ is a cone on a $d$-regular graph on $n$ vertices, then $|U_{A(X)}(t)_{u,u}|\geq \frac{d^2}{d^2+4n}$ with equality if and only if $t=\frac{\pi}{\sqrt{d^2+4n}}$ \cite{SedQW}. Thus, (1a) is immediate. In (1b), we have $n=(k-2)m+2$ and we may write $X=K_2\vee Y$, where $Y=\bigvee_{k-2} O_m$ is a $(n-m-2)$-regular graph on $n-2$ vertices. Invoking Corollary \ref{yipee1} and Theorems 12(1) and 14(2b) in \cite{Kirkland2023} yields (1b). Now, (1ci) is immediate from Theorem \ref{sed2}(1,2). To prove (1cii), note that $X=K_{|\mathcal{I}|}\vee K_{m,\ldots,m}$, and so $\sigma_u(A(X))=\{-1,\frac{1}{2}(d\pm\sqrt{\Delta})\}$ by \cite[Lemma 3(2)]{kirkland2023quantum}, where $d=n-m-1$. Let $S=\{-1\}$ and $a=\frac{|\mathcal{I}|-1}{|\mathcal{I}|}$ in Lemma \ref{eurekarem2}. If $a,b,c$ are integers such that $a[\frac{1}{2}(d+\sqrt{\Delta})]+b[\frac{1}{2}(d-\sqrt{\Delta})]+(-1)c=0$, then $a=b$ and $c=a(n-m-1)$ because $\Delta$ is not a perfect square. Thus, $a+b+c=a(n-m+1)=0$ for any integer $a$ if and only if $n=m-1$, a contradiction. Thus, Lemma \ref{eurekarem2} yields (1cii). Lastly in (1d), we have $X= K_{|\mathcal{I}|}\vee CP(n-|\mathcal{I}|)$, where $n-|\mathcal{I}|$ is even and $\sigma_w(A(CP(n-|\mathcal{I}|)))=\{0,-2,n-2\}$ for each vertex $w$ of $CP(n-|\mathcal{I}|)$ \cite[Example 2]{Kirkland2023}. Invoking \cite[Lemma 3(2)]{kirkland2023quantum}, we get $\sigma_w(A(X))=\{0,-2,\frac{1}{2}(n-3\pm\sqrt{\Delta})\}$. Since $w$ is strongly cospectral with its twin in $CP(n-|\mathcal{I}|)$, say $v$, \cite[Corollary 20(2)]{kirkland2023quantum} implies that $w$ and $v$ are also strongly cospectral in $X$ with $\sigma_{uv}^+(A(X))=\{-2,\frac{1}{2}(n-3\pm\sqrt{\Delta})\}$ and $\sigma_{uv}^-(A(X))=\{0\}$. To prove (1di), let $a,b,c$ be integers such that $a[\frac{1}{2}(n-3+\sqrt{\Delta})]+b[\frac{1}{2}(n-3-\sqrt{\Delta})]+(-2)c=0$. Since $\Delta$ is not a perfect square, we get $a=b$ and $c=\frac{a(n-3)}{2}$. In this case, $a+b+c=a\left(2+\frac{n-3}{2}\right)$ is even for any integer $a$ if and only if $n\equiv 3$ (mod 4). Applying \cite[Theorem 13]{Kirkland2023} yields the desired result for (1di), while (1dii) and (1diii) are resp.\ immediate from \cite[Theorem 11(1)]{kirkland2023quantum} and Corollary \ref{yipee1}. 
    
    In (2), we have $X=O_2\vee Y$, and so a direct application of \cite[Theorem 36]{Monterde2023} and \cite[Theorem 11(1)]{Kirkland2023} yields the desired result. In particular, the assumption in (2b) is equivalent to $\Delta$ is a perfect square and $\nu_2(n-m-2)=\nu_2(\sqrt{\Delta})$). Finally, (3) is immediate from Theorem \ref{sed2}(1,2).
\end{proof}

\begin{remark}
Let $X=CP(2k)$. Taking $n=2k$, $s=2$, and $m=2$ in Theorem \ref{cmgA}(2) yields $n-m-2=2k-4$ and $\sqrt{\Delta}=2k$. If $k$ is even, then $\nu_2(2k-4)\neq \nu_2(2k)$, and so PST occurs in $X$ by Theorem \ref{cmgA}(2a). If $k$ is odd, then $\nu_2(2k-4)= \nu_2(2k)$ and $g=2$, so that %$s_1=1$ and $d_1=k-2$. 
each $u\in V(X)$ is tightly $(\frac{1}{k})$-sedentary at $t=\frac{j\pi}{g}$ for any odd $j$ by Theorem \ref{cmgA}(2bi). This coincides with Theorem \ref{cmg11} since $X$ is regular.
%for the Laplacian case 
\end{remark}

We now present an analogue of Corollary \ref{cmg11kne} for the adjacency case.

\begin{corollary}
\label{cmg11kneA}
For all $n\geq 3$, let $X=K_n\backslash e$ with non-adjacent vertices $u$ and $v$.
\begin{enumerate}
\item If $n=3$, then perfect state transfer occurs between $u$ and $v$ in $X$ and the degree two vertex of $X$ is not sedentary.
\item If $n=4$, then the two degree three vertices of $X$ admit pretty good state transfer. Moreover, for all $n\geq 5$, each degree $n-1$ vertex of $X$ is sharply $(1-\frac{2}{n-2})$-sedentary.
\item For all $n\geq 4$, pretty good state transfer occurs between $u$ and $v$ in $X$.
\end{enumerate}
\end{corollary}

\begin{proof}
    Note that $\Delta=n^2+2n-7$ is not a perfect square for all $n\geq 3$. Now, (1) follows from PST in $P_3$ and Example \ref{star}. The first statement in (2) follows from \cite[Theorem]{Kirkland2023} and the fact that $\sigma_w(A(X))=\{-1,\frac{1}{2}(1\pm\sqrt{17})\}$ for a degree three vertex $w$ in $X$, while the second statement is obtained by applying Theorem \ref{cmgA}(1c) with $|\mathcal{I}|=n-2$ and $m=2$. Applying Theorem \ref{cmgA}(2a) with $m=1$ yields (3).
\end{proof}

\section{Threshold graphs}\label{secTG}

A connected \textit{threshold graph} is a graph that is either of the form:
\begin{itemize}
    \item $((((O_{m_1}\vee K_{m_2})\cup O_{m_3})\vee K_{m_4})\ldots)\vee K_{m_{2k}}:=\Gamma(m_1,\ldots,m_{2k})$
    \item $((((K_{m_1}\cup O_{m_2})\vee K_{m_3})\cup O_{m_4})\ldots)\vee K_{m_{2k+1}}:=\Gamma(m_1,\ldots,m_{2k+1})$.
\end{itemize}
PST and fractional revival in threshold graphs under Laplacian dynamics have been investigated in \cite{Severini} and \cite{Kirkland2020}, respectively. Here, we characterize Laplacian sedentariness in this family of graphs.

\begin{theorem}
\label{psed}
    Let $X\in\{\Gamma(m_1,\ldots,m_{2k}),\Gamma(m_1,\ldots,m_{2k+1})\}$. If $m_1=2$, $m_2\equiv 2$ (mod 4) and $m_j\equiv 0$ (mod 4) for all $j\neq 1,2$, then perfect state transfer occurs between the two vertices of $X_1\in\{O_{m_1},K_{m_1}\}$. Otherwise, each vertex in $X$ is sedentary.
\end{theorem}

\begin{proof}
The first statement follows from \cite[Theorem 2]{Severini}. The second follows from Corollary \ref{yipee2}, $L(X)$ having integer eigenvalues and the absence of strongly cospectrality in $X$ \cite[Corollary 67]{kirkland2023quantum}.
\end{proof}

Next, we provide bounds on the sedentariness of a vertex in a threshold graph.

\begin{theorem}
\label{thresh}
    Let $X\in\{\Gamma(m_1,\ldots,m_{2k}),\Gamma(m_1,\ldots,m_{2k+1})\}$. Suppose $j\in\{1,\ldots,2k+1\}$ and $u$ is a vertex of $Y_j\in\{O_{m_j},K_{m_j}\}$. Let $\alpha_j=\sum_{r=1}^jm_r$. The following hold.
 \begin{enumerate}
 \item Let (a) $j$ be even, $X=\Gamma(m_1,\ldots,m_{2k})$ and $h=2k$ or (b) $j\geq 3$ be odd, $X=\Gamma(m_1,\ldots,m_{2k+1})$ and $h=2k+1$. For each integer $\ell$ having the same parity as $h$ such that $j+2\leq \ell \leq h$, define $\beta_{\ell,h}:=m_{\ell}+m_{\ell+2}+\ldots+m_h$. Then
 \begin{center}
     $\displaystyle |U_{L(X)}(t)_{u,u}|\geq 1-2\sum_{r=j}^{h}\frac{(-1)^{r+h}}{\alpha_r}\quad \text{for all $t$},$
 \end{center}
with equality if and only if for some $t$, $e^{it\alpha_{h}}=-1$ and $e^{it\beta_{\ell,h}}=1$ for each even $j+2\leq \ell\leq h$.
\item Let $j=1$, $m_1\geq 2$ and $X=\Gamma(m_1,\ldots,m_{2k+1})$. For each odd $3\leq \ell\leq 2k+1$, define $\beta_{\ell}=\sum_{r=0}^{(2k+1-\ell)/2}m_{\ell+2r}=m_{\ell}+m_{\ell+2}+\ldots+m_{2k+1}$. Then
 \begin{equation}
 \label{2}
 |U_{L(X)}(t)_{u,u}|\geq 1-2/m_1\quad \text{for all $t$},
\end{equation}
with equality if and only if for some $t$, $e^{it(m_1+\beta_3)}=-1$, $e^{it\alpha_{2k+1}}=e^{it(\alpha_{\ell}+\beta_{\ell+2})}=1$ for each odd $3\leq \ell\leq 2k-1$, and $e^{it\beta_{\ell+2}}=1$ for each odd $1\leq \ell\leq 2k-1$.
%$e^{it\alpha_{2k+1}}=e^{it(\alpha_{\ell}+\beta_{\ell+2})}=-1$ and $e^{it\beta_{\ell+2}}=1$ for each odd $1\leq \ell\leq 2k-1$.
\item Let (a) $j$ be odd, $X=\Gamma(m_1,\ldots,m_{2k})$ and $h=2k$ or (b) $j$ be even, $X=\Gamma(m_1,\ldots,m_{2k+1})$ and $h=2k+1$. For each integer $\ell$ that has the same parity as $h$ such that $j\leq \ell\leq h$, define $\beta_{\ell}=\sum_{r=0}^{(2k-\ell)/2}m_{\ell+2r}=m_{\ell}+m_{\ell+2}+\ldots+m_{h}$. Then
 \begin{equation}
 \label{3}
 |U_{L(X)}(t)_{u,u}|\geq 1-2/\alpha_j\quad \text{for all $t$},
\end{equation}
with equality if and only if for some $t$, $e^{it\beta_{j+1}}=-1$, $e^{it\alpha_{h}}=e^{it\beta_{\ell}}=1$ for each $j+2\leq \ell\leq h$ and $e^{it(\alpha_{\ell}+\beta_{\ell+2})}=1$ for each $j\leq \ell\leq h$, where $\ell$ has the same parity as $h$.
    \end{enumerate}
\end{theorem}

\begin{proof}
    We first prove (1). For each even $j\leq h\leq 2k$, let $X_{h}=((((O_{m_1}\vee K_{m_2})\cup O_{m_3})\vee K_{m_4})\ldots)\vee K_{m_{h}}$ (so that $X_{2k}=X$). Since $\sigma_u(L(K_{m_j}))=\{0,m_j\}$, repeatedly applying \cite[Lemma 3(2)]{kirkland2023quantum} $(h-j)/2$ times starting from $K_{m_j}$ to $X_j$ all the way to $X_h$ gives us
    \begin{equation}
    \label{sup}
        \sigma_u(L(X_{h}))=\{\beta_{\ell,h}:j+2\leq \ell\leq h\}\cup\{0,\alpha_h\}.
    \end{equation}
    Since $[m_{h+2}\textbf{1},-\alpha_h\textbf{1}]^T$ is an eigenvector for $L(X_{h+2})$ associated with $\alpha_{h+2}$, the multiplicity of $\alpha_{h+2}$ is equal to that of $\alpha_h$ plus one. Thus, $(E_{\alpha_{h+2}})_{u,u}=(E_{\alpha_{h}})_{u,u}+(\frac{1}{\alpha_{h+1}}-\frac{1}{\alpha_{h+2}})$.  Arguing inductively, we get $(E_{\alpha_{2k}})_{u,u}=1-\sum_{r=j}^{2k}\frac{(-1)^r}{\alpha_r}$ in $L(X)$.
    Taking $S=\{\alpha_{2k}\}$ and $a=(E_{\alpha_{2k}})_{u,u}$ in Theorem \ref{eureka} yields (1a). For (1b), the same argument applied to $X_{h}=((((K_{m_1}\cup O_{m_2})\vee K_{m_3})\cup O_{m_4})\ldots)\vee K_{m_{h}}$ for each odd $j\leq h\leq 2k+1$ yields the desired result.
    
    To prove (2), note that $\alpha_1=m_1$. The same argument used in (\ref{sup}) yields
    \begin{center}
        $\sigma_u(L(X))=\{\beta_{\ell}:3\leq \ell\leq 2k+1\ \text{is odd}\}\cup\{\alpha_{\ell}+\beta_{\ell+2}:1\leq \ell\leq 2k-1\ \text{is odd}\}\cup\{0,\alpha_{2k+1}\}$.
    \end{center}
    Note that $\sigma_u(L(K_{m_1}))=\{0,\alpha_1\}$ and $\alpha_1+\beta_{3,h}\in\sigma_u(L(X_h))$ for each odd $3\leq h\leq 2k+1$, where $\beta_{3,h}$ is defined in (1) and $X_{h}=((((K_{m_1}\cup O_{m_2})\vee K_{m_3})\cup O_{m_4})\ldots)\vee K_{m_{h}}$. Since
    $E_{\alpha_1}=1-\frac{1}{\alpha_1}$ and the multiplicity of $\alpha_1+\beta_{3,h}\in\sigma_u(L(X_h))$ is equal to that of $\alpha_1$, we get $(E_{\alpha_{1}+\beta_{3}})_{u,u}=E_{\alpha_1}=1-\frac{1}{\alpha_1}$. Taking $S=\{\alpha_1+\beta_3\}$ and $a=(E_{\alpha_{1}+\beta_3})_{u,u}=1-\frac{1}{\alpha_1}$ in Theorem \ref{eureka} yields (2).
   
   To prove (3a), suppose $j$ is odd and $X=\Gamma(m_1,\ldots,m_{2k})$. The same argument used in (\ref{sup}) yields
    \begin{center}
        $\sigma_u(L(X))=\{\beta_{\ell}:j+1\leq \ell\leq 2k\ \text{is even}\}\cup\{\alpha_{\ell}+\beta_{\ell+2}:j+1\leq \ell\leq 2k-2\ \text{is even}\}\cup\{0,\alpha_{2k}\}.$
    \end{center}
    For each odd $j\leq h\leq 2k$, let $X_{h}=(((((O_{m_1}\vee K_{m_2})\cup O_{m_3})\vee K_{m_4})\ldots)\cup O_{m_h})\vee K_{m_{h+1}}$. Note that $\sigma_u(L(O_{m_j}))=\{0\}$
    and $\beta_{j+1,h}\in\sigma_u(L(X_h))$ for each odd $j\leq h\leq 2k$, where $\beta_{j+1,h}$ is defined in (1). From the form of the $X_h$'s, we deduce that the multiplicity of $\beta_{j+1,h+2}\in\sigma_u(L(X_{h+2}))$ is equal to that of $\beta_{j+1,h}\in\sigma_u(L(X_{h}))$. Since $(E_{\beta_{j+1,j}})_{u,u}=\left(1-\frac{1}{m_j}\right)+\left(\frac{1}{m_j}-\frac{1}{\alpha_j}\right)=1-\frac{1}{\alpha_j}$ in $L(X_{j})$, we get $(E_{\beta_{j+1,h}})_{u,u}=1-\frac{1}{\alpha_j}$ in $L(X_{h})$ for each odd $j\leq h\leq 2k$. Letting $h=2k-1$, $S=\{\beta_{j+1}\}$ and $a=(E_{\beta_{j+1}})_{u,u}=1-\frac{1}{\alpha_j}$ in Theorem \ref{eureka} yields the desired conclusion in (3a). The same argument works for (3b).
\end{proof}

\begin{remark}
If $m_1=1$ and $X=\Gamma(m_1,\ldots,m_{2k})$, then $X=\Gamma(m_1',\ldots,m_{2k-1}')$, where $m_1'=m_2+1$ and $m_j'=m_{j+1}$ for all $j\geq 1$. In this case, Theorem \ref{thresh}(2) applies. Similarly, if $m_1=1$ and $X=\Gamma(m_1,\ldots,m_{2k+1})$, then Theorem \ref{thresh}(3) applies.
\end{remark}

\begin{remark}
Let $m_1=2$ so that $\alpha_1=2$ in Theorem \ref{thresh}(2-3). If $m_2\equiv 2$ and $m_r\equiv 0$ (mod 4) for all $r\neq 1,2$, then (\ref{2}) and (\ref{3}) yields $U_{L(X)}(\frac{\pi}{2})_{u,u}=0$ for $X\in \{\Gamma(m_1,\ldots,m_{2k+1}),\Gamma(m_1,\ldots,m_{2k})\}$. This is consistent with the fact that $u$ is involved in PST at $t=\frac{\pi}{2}$ by Theorem \ref{psed}. However, if $m_2\not\equiv 2$ or $m_r\not\equiv 0$ (mod 4) for some $r\neq 1,2$, then there is no time $t$ such that the lower bounds in (\ref{2}) and (\ref{3}) are attained respectively. In this case, $u$ is sedentary in $X$.
\end{remark}

Lastly, we have the following asymptotic result  that applies to $M_q$ for all $q$.

\begin{corollary}
\label{almost}
Almost all complete multipartite graphs and threshold graphs contain a sedentary vertex.
\end{corollary}

\begin{proof}
The number of non-isomorphic complete bipartite graphs on $n$ vertices is equal to $P(n)$, the number of partitions of $n$ (where the partitions represent the sizes of the partite sets). Similarly, the number of non-isomorphic complete bipartite graphs on $n$ vertices which contain a partite set of size three is equal to $P(n-3)$. From Theorem \ref{sed2}(2), we know that each vertex in a partite set of size three is sedentary. Using the famous asymptotic formula $P(n)\sim \frac{1}{4n\sqrt{3}}e^{\pi\sqrt{2n/3}}$ as $n\rightarrow\infty$ due to Hardy and Ramanujan \cite{HardyAsymptoticFI}, it follows that $P(n-3)/P(n)\rightarrow 1$ as $n\rightarrow\infty$. The same argument applies to threshold graphs.
\end{proof}

\section{Direct Product}\label{secDP}

In \cite[Theorem 4]{Monterde2023}, it was shown that the Cartesian product preserves sedentariness. In this section, we investigate whether the direct product also preserves sedentariness relative to the adjacency matrix.

The \textit{direct product} $X\times Y$ of $X$ and $Y$ is the graph with vertex set $V(X)\times V(Y)$ and adjacency matrix
\begin{equation*}
A(X)\otimes A(Y).
\end{equation*}
If $A(X)=\sum_{j}\lambda_j E_j$ is a spectral decomposition of $A(X)$, then it was shown in \cite[Lemma 4.2]{coutinho2016perfect} that
\begin{equation}
\label{dirprod1}
U_{A(X\times Y)}(t)=\sum_{j}E_j\otimes U_{A(Y)}(\lambda_jt).
\end{equation}
Since $A(K_m)=(m-1)(\frac{1}{m}\textbf{J}_m)+(-1)(I_m-\frac{1}{m}\textbf{J}_m)$, (\ref{dirprod1}) yields the following transition matrix for $K_m\times Y$
\begin{equation}
\label{dirprod2}
\begin{split}
U_{A(K_m\times Y)}(t)&=\frac{1}{m}\textbf{J}_m\otimes U_{A(Y)}((m-1)t) +(I_m-\frac{1}{m}\textbf{J}_m)\otimes U_{A(Y)}(-t).%\\&=\frac{1}{m}\textbf{J}_m\otimes U_{A(Y)}(mt)U_{A(Y)}(-t)+(I_m-\frac{1}{m}\textbf{J}_m)\otimes U_{A(Y)}(-t).
\end{split}
\end{equation}
Recall that $K_2\times Y$ is called the \textit{bipartite double} of $Y$, and this graph is connected if and only if $Y$ is non-bipartite. Letting $m=2$ and taking the diagonal entries in (\ref{dirprod2}) gives us
%the diagonal entries of the transition matrix of $K_2\times Y$: 
\begin{equation}
\label{dirprod3}
\begin{split}
\left|U_{A(K_2\times Y)}(t)_{((u,v),(u,v)}\right|=|\operatorname{Re}(U_{A(Y)}(t)_{v,v})|.
\end{split}
\end{equation}

\begin{theorem}
\label{dirprod}
Let $u\in V(K_m)$ and $v\in V(Y)$. The following hold.
\begin{enumerate}
\item If $m\geq 3$ and $Y$ is $C$-sedentary at vertex $v$, where $C> \frac{1}{m-1}$, then $K_m\otimes Y$ is $(C-\frac{C+1}{m})$-sedentary at vertex $(u,v)$ for any vertex $u$ of $K_m$.
\item $K_2\times Y$ is $C$-sedentary at $(u,v)$ if and only if $|\operatorname{Re}(U_{A(Y)}(t)_{v,v})|\geq C$ for all $t$.
\end{enumerate}
\end{theorem}

\begin{proof}
Taking the $((u,v),(u,v))$-entry of (\ref{dirprod2}) and applying triangle inequality yields
\begin{equation*}
\begin{split}
\left|U_{A(K_m\times Y)}(t)_{((u,v),(u,v)}\right|&=\left|\frac{1}{m}U_{A(Y)}((m-1)t)_{v,v}+\frac{m-1}{m}U_{A(Y)}(-t)_{v,v}\right|\\
&\geq \frac{m-1}{m}\left|U_{A(Y)}(-t)_{v,v}\right|-\frac{1}{m}\left|U_{A(Y)}((m-1)t)_{v,v}\right|\\
&\geq \frac{m-1}{m}\left|U_{A(Y)}(t)_{v,v}\right|-\frac{1}{m}.
%\left|U_{A(Y)}((m-1)t)_{v,v}\right|\quad \text{for all $t$}.
\end{split}
\end{equation*}
Therefore, if $\left|U_{A(Y)}(t)_{v,v}\right|\geq C$ for all $t$, then $\left|U_{A(K_m\vee Y)}(t)_{((u,v),(u,v)}\right|\geq \frac{C(m-1)}{m}-\frac{1}{m}= C-\frac{C+1}{m}$. From this, (1) is immediate, and (2) follows from (\ref{dirprod3}).
\end{proof}

\begin{example}
\label{kmxY}
Let $T$ be a twin set in $Y$. If $|T|\geq 3$ and $C=1-\frac{2}{|T|}>\frac{1}{m-1}$, then Theorems \ref{sed2}(2) and \ref{dirprod}(1) imply that $K_m\otimes Y$ is $(C-\frac{C+1}{m})$-sedentary at $(u,v)$ for each $u\in V(K_m)$ and each $v\in T$.
\end{example}

Next, we examine direct products of complete graphs. Let $[n]=\{1,2,\ldots,n\}$. Denote $Y_1\times Y_2\times \ldots\times Y_n$ by $\bigtimes_{j\in [n]} Y_j$ and the direct product of $n$ copies of $Y$ by $Y^{\times n}$. For the graph $Y_1\times Y_2\times \ldots\times Y_n$, the ordering of indices $1,2,\ldots,n$ does not matter since the direct product is commutative up to isomorphism. In our next result, we determine the diagonal entries of the transition matrix of $\bigtimes_{j\in [n]} K_{m_j}$.

\begin{theorem}
\label{complete}
Let $n\geq 2$. For each $j\in [n]$, let $m_j\geq 2$ be an integer and consider $X=\bigtimes_{j\in[n]}K_{m_j}$. Then
\begin{equation}
\label{induct}
U_{X}(t)_{((v_1,\ldots,v_n),(v_1,\ldots,v_n))}=\frac{1}{\prod_{j\in [n]} m_j}\displaystyle\sum_{S\subseteq [n]}\left(\prod_{j\in S}(m_j-1)\right)e^{it(-1)^{|S|}\prod_{j\notin S}(m_j-1)},
\end{equation}
where $\prod_{j\in S}(m_j-1)=1$ whenever $S=\varnothing$. Moreover, the following hold.
\begin{enumerate}
\item If $m_j=2$ for some $j\in [n]$, then
\begin{equation}
\label{induct1}
U_{X}(t)_{((v_1,\ldots,v_n),(v_1,\ldots,v_n))}=\frac{1}{\prod_{j\in [n]\backslash\{j\}} m_j}\displaystyle\sum_{S\subseteq [n]\backslash\{j\}}\left(\prod_{j\in S}(m_j-1)\right)\cos\left(t\prod_{j\notin S}(m_j-1)\right).
\end{equation}
Thus, if there is a time $t_1$ such that $\cos\left(t_1\prod_{j\notin S}(m_j-1)\right)<0$ for each $S\subseteq [n]\backslash\{j\}$, then each vertex of $X$ is not sedentary. In particular, if each $m_j$ is even, then each vertex of $X$ is not sedentary.
\item If $\displaystyle\prod_{j\in [n]}(m_j-1)\neq \frac{1}{2}\prod_{j\in[n]} m_j$, then each vertex of $X$ is $\displaystyle\frac{2}{\prod_{j\in [n]} m_j}\bigg|\prod_{j\in [n]}(m_j-1)-\frac{1}{2}\prod_{j\in[n]} m_j\bigg|$-sedentary, and this is tight at time $t=\pi$ whenever each $m_j$ is odd. In particular, if $m\geq 2$, then each vertex of $K_m^{\times n}$ is $\displaystyle\frac{2}{m^n}\bigg|(m-1)^n-\frac{1}{2}m^n\bigg|$-sedentary.
\end{enumerate}
\end{theorem}

\begin{proof}
We prove (\ref{induct}) by induction. The base case follows from the fact that $U_{A(K_n)}(t)_{u,u}=e^{it(n-1)}(\frac{1}{n}\textbf{J}_n)+e^{-it}(I_n-\frac{1}{n}\textbf{J}_n)$. For the induction step, let $k\geq 1$ be an integer and suppose (\ref{induct}) holds whenever $n=k$. Let $Y=K_{m_{k+1}}\times \bigtimes_{j\in[k]}K_{m_j}$. Applying (\ref{dirprod2}) to $Y$, and simplifying the resulting transition matrix yields
\begin{equation*}
\begin{split}
U_{Y}(t)_{((v_1,\ldots,v_n),(v_1,\ldots,v_n))}&=\frac{1}{m_{k+1}}\left[\frac{1}{\prod_{j\in [k]} m_j}\displaystyle\sum_{S\subseteq [k]}\left(\prod_{j\in S}(m_j-1)\right)e^{it(-1)^{|S|}\prod_{j\notin S}(m_j-1)(m_{k+1}-1)}\right]\\
&+\frac{m_{k+1}-1}{m_{k+1}}\left[\frac{1}{\prod_{j\in [k]} m_j}\displaystyle\sum_{S\subseteq [k]}\left(\prod_{j\in S}(m_j-1)\right)e^{it(-1)^{|S|+1}\prod_{j\notin S}(m_j-1)}\right]\\
&=\frac{1}{\prod_{j\in [k+1]} m_j}\displaystyle\sum_{S\subseteq [k]}\left(\prod_{j\in S}(m_j-1)\right)e^{it(-1)^{|S|}\prod_{j\notin S}(m_j-1)(m_{k+1}-1)}\\
&+\frac{1}{\prod_{j\in [k+1]} m_j}\displaystyle\sum_{S\cup \{k+1\}\subseteq [k+1]}\left(\prod_{j\in S\cup \{k+1\}}(m_j-1)\right)e^{it(-1)^{|S|+1}\prod_{j\notin S}(m_j-1)}.
\end{split}
\end{equation*}
Observe that the first summand in the second equality above is a summation that runs over all subsets of $[k+1]$ that do not contain $k+1$, whereas the second summand runs over all subsets of $[k+1]$ that contain $k+1$. Thus, combining the two summands yields (\ref{induct}) with $n=k+1$. This establishes (\ref{induct}) for all $n\geq 1$.

To prove (1), we may without loss of generality assume that $m_{n}=2$. Making use of the last equality above with $n=k+1$ yields
\begin{equation*}
\begin{split}
U_{X}(t)_{((v_1,\ldots,v_n),(v_1,\ldots,v_n))}&=\frac{1}{2\prod_{j\in [n-1]} m_j}\displaystyle\sum_{S\subseteq [n-1]}\left(\prod_{j\in S}(m_j-1)\right)e^{it(-1)^{|S|}\prod_{j\notin S}(m_j-1)}\\
&+\frac{1}{2\prod_{j\in [n-1]} m_j}\displaystyle\sum_{S\subseteq [n-1]}\left(\prod_{j\in S}(m_j-1)\right)e^{it(-1)^{|S|+1}\prod_{j\notin S}(m_j-1)}\\
%&=\frac{1}{\prod_{j\in [n-1]} m_j}\displaystyle\sum_{S\subseteq [n-1]}\left(\prod_{j\in S}(m_j-1)\right)\cos\left(t\prod_{j\notin S}(m_j-1)\right).
\end{split}
\end{equation*}
Adding the two summands above yields (\ref{induct1}). Now, note that $U_{X}(0)_{((v_1,\ldots,v_n),(v_1,\ldots,v_n))}>0$ from (\ref{induct1}). Thus, if $\cos\left(t_1\prod_{j\notin S}(m_j-1)\right)<0$ for each $S\subseteq [n-1]$, then $U_{X}(t_1)_{((v_1,\ldots,v_n),(v_1,\ldots,v_n))}<0$. Invoking the intermediate value theorem, there exists $t_0>0$ such that $U_{Y}(t_0)_{((v_1,\ldots,v_n),(v_1,\ldots,v_n))}=0$, and so each vertex of $X$ is not sedentary by Theorem \ref{dirprod}(2). In particular, if $m_j$ is even for each $j\in [n]$, then we may take $t_1=\pi$, and so the conclusion applies.

Finally, we prove (2). Rewriting (\ref{induct}) and applying triangle inequality gives us
\begin{equation*}
\begin{split}
|U_{X}(t)_{((v_1,\ldots,v_n),(v_1,\ldots,v_n))}|&=\frac{1}{\prod_{j\in [n]} m_j}\left|\prod_{j\in [n]}(m_j-1)e^{it(-1)^{n}}+\displaystyle\sum_{S\subsetneq [n]}\left(\prod_{j\in S}(m_j-1)\right)e^{it(-1)^{|S|}\prod_{j\notin S}(m_j-1)}\right|\\
&=\frac{1}{\prod_{j\in [n]} m_j}\left|\prod_{j\in [n]}(m_j-1)+\displaystyle\sum_{S\subsetneq [n]}\left(\prod_{j\in S}(m_j-1)\right)e^{it(-1)^{|S|}\left(\prod_{j\notin S}(m_j-1)-(-1)^{n-|S|}\right)}\right|\\
&\geq \frac{1}{\prod_{j\in [n]} m_j}\left|\prod_{j\in [n]}(m_j-1)-\displaystyle\sum_{S\subsetneq [n]}\left(\prod_{j\in S}(m_j-1)\right)\right|\\
&=\frac{2}{\prod_{j\in [n]} m_j}\left|\prod_{j\in [n]}(m_j-1)-\frac{1}{2}\displaystyle\prod_{j\in[n]} m_j\right|
\end{split}
\end{equation*}
for all $t$, where the last equality above uses the fact that $\displaystyle\prod_{j\in[n]} m_j=\prod_{j\in[n]} (m_j-1)+\sum_{ S\subsetneq [n]}\left(\prod_{j\in S}(m_j-1)\right)$. Moreover, the above inequality is an equality at time $t=\pi$ whenever $m_j$ is odd for each $j\in[n]$. From these considerations, statement (2) is immediate.
\end{proof}

\begin{corollary}
\label{kmxY1}
Let $u\in V(K_m)$ and $v\in V(K_n)$. The following hold in $K_m\times K_n$.
\begin{enumerate}
\item If $m=2$ or $n=2$, then $(u,v)$ is not sedentary.
\item If $m,n\geq 3$, then $(u,v)$ is sedentary. In particular, if $(m,n)\neq (3,4)$, then $(u,v)$ is $(|\frac{mn-2(m+n-1)}{mn}|)$-sedentary and this is tight whenever $m$ and $n$ are odd. Otherwise, $(u,v)$ is tightly $0.142$-sedentary.
\end{enumerate}
\end{corollary}

\begin{proof}
Let $m=2$. By (\ref{induct1}), $U_{A(K_2\times K_n)}(t)_{((u,v),(u,v)}=\frac{1}{n}\cos(t(n-1))+\frac{n-1}{n}\cos t$. If $n=2$, then $K_2\times K_2$ is disconnected and each vertex pairs up with another to admit PST. Now, suppose $n\geq 3$. Since $\frac{1}{n}\cos(t(n-1))+\frac{n-1}{n}\cos t<0$ at $t=\pi$, we conclude that each vertex of $K_2\times K_n$ is not sedentary by Theorem \ref{complete}(1). Thus, (1) holds. To prove (2), note that the case $(m,n)\neq (3,4)$ is immediate from Theorem \ref{complete}(2). Now, if $(m,n)= (3,4)$, then we get $\left|U_{A(K_3\times K_4)}(t)_{((u,v),(u,v)}\right|=\frac{1}{12}\left|e^{5it}+3e^{-3it}+2e^{-4it}+6\right|$. Using your favourite computer package, you may verify that $\left|U_{A(K_3\times K_4)}(t)_{((u,v),(u,v)}\right|\geq 0.142$ with equality whenever $t\in\{0.945,2\pi-0.945\}$. Thus, each vertex of $K_m\times K_n$ is sedentary for all $m,n\geq 3$.
\end{proof}

\begin{remark}
\label{exsc}
Consider $X=K_2\times K_n$ and let $V(K_2)=\{1,2\}$. By Corollary \ref{kmxY1}(1), $X$ is not sedentary at every vertex. Now, for each $u\in V(K_n)$, one checks that vertex $(1,u)$ is strongly cospectral only to vertex $(2,u)$ in $X$. Invoking \cite[Theorem 3.3.4]{Coutinho2014}, we obtain PST between $(1,u)$ and $(2,u)$ in $X$ if and only if $n$ is even. Consequently, each vertex of $K_2\times K_n$ for each odd $n$ is periodic, sedentary, and is involved in strong cospectrality but not PST. This also applies to the Laplacian case because $X$ is regular.
\end{remark}

To complement Corollary \ref{kmxY1}(1), we provide an infinite family of graphs of the form $K_2\times Y$ that contain sedentary vertices, where $Y$ is non-bipartite.

\begin{theorem}
\label{k2xYsed}
Let $Z$ be a $d$-regular graph on $n$ vertices such that $d>0$ is an integer and $n=\frac{1}{2}s(d+s)$ for some even integer $s$ satisfying $\nu_2(s)\geq \nu_2(d)$. If $v$ is an apex of $Y=O_2\vee Z$, then $K_2\times Y$ is tightly $C$-sedentary at vertex $(u,v)$ for any vertex $u$ of $K_m$, where
%$d^2+8n$ is a perfect square, $d>0$, and $\nu_2(d+\sqrt{d^2+8n})\neq \nu_2(d-\sqrt{d^2+8n})$.
\begin{center}
$C=\min_{k\in\mathcal{K}}\cos^2\left(\frac{s_1k\pi}{d_1+2s_1}\right)>0$,
\end{center}
$\mathcal{K}=\left\{k\in\mathbb{Z}^+:\lceil \frac{j(d_1+2s_1)}{4s_1}\rceil\leq k\leq \lfloor \frac{(j+2)(d_1+2s_1)}{4s_1} \rfloor,\ j\leq s_1,\ \text{$j\equiv 1$ (mod 4)}\right\}$, $d_1=\frac{d}{\operatorname{gcd}(d,s)}$ and $s_1=\frac{s}{\operatorname{gcd}(d,s)}$. In particular, $\left|U_{A(K_2\times Y}(\tau)_{((u,v),(u,v)}\right|=C$ at time $\tau=\frac{2k_0\pi}{d+2s}$, where $k_0\in\mathcal{K}$ such that $C=\cos^2\left(\frac{s_1k_0\pi}{d_1+2s_1}\right)$.
\end{theorem}

\begin{proof}
From the proof of \cite[Theorem 36]{Monterde2023}, $U_{A(Y)}(t)_{u,u}=\frac{1}{2(d+2s)}\left((d+2s)+se^{it(d+s)}+(d+s)e^{-its}\right)$. From this, we obtain
%and vertex $v$ is periodic in $Y$ with minimum period $2\pi/g$, where $g=\operatorname{gcd}(d,s)$
\begin{equation}
\label{re}
\operatorname{Re}\left(U_{A(Y)}(t)_{v,v}\right)=\frac{1}{2(d+2s)}\left((d+2s)+f(t)\right),
\end{equation}
where $f(t)=s\cos(t(d+s))+(d+s)\cos(ts)$. Note that $f'(t)=-s(d+s)\left(\sin(t(d+s)+\sin(ts)\right)$ and $f''(t)=-s(d+s)\left((d+s)\cos(t(d+s)+s\cos(ts)\right)$. Using a sum to product identity, we obtain $f'(t)=-2s(d+s)\left(\sin(t(d+2s)/2)\cos(td/2)\right)$. Therefore $f'(t)=0$ if and only if either $t=\frac{2k\pi}{d+2s}$ for any integer $k$ or $t=\ell\pi/d$ for any odd integer $\ell$.
%Let $d=gd_1$ and $s=gs_1$ so that
\begin{itemize}
\item Let $t=\frac{2k\pi}{d+2s}$, where $0<k\leq d+2s$ is an integer. Then $\cos(ts)=\cos(t(d+s))=\cos(\frac{2s_1k\pi}{d_1+2s_1})$, and so $f''(t)=-s(d+s)(d+2s)\cos(\frac{2s_1k\pi}{d_1+2s_1})>0$ if and only if $\cos(\frac{2s_1k\pi}{d_1+2s_1})<0$, i.e., $\frac{j\pi}{2}<\frac{2s_1k\pi}{d_1+2s_1}<\frac{(j+2)\pi}{2}$ for each $j\leq s_1$ with $j\equiv 1$ (mod 4). Thus, $\lceil \frac{j(d_1+2s_1)}{4s_1}\rceil\leq k\leq \lfloor \frac{(j+2)(d_1+2s_1)}{4s_1} \rfloor$ and  $f(t)=(d+2s)\cos(\frac{2s_1k\pi}{d_1+2s_1})$ attains a minima. Using (\ref{re}), we get $\operatorname{Re}\left(U_{A(Y)}(t)_{v,v}\right)=\cos^2(\frac{s_1k\pi}{d_1+2s_1})$,
which is positive because $d_1+2s_1$ is odd by assumption, and so $\frac{s_1k\pi}{d_1+2s_1}\neq \frac{j\pi}{2}$ for each any odd $j$.
\item Let $t=\ell\pi/d$ for odd $\ell$. Then $\cos(ts)=\cos(\frac{\ell s_1\pi}{d_1})=-\cos(t(d+s))$, and so $f''(t)=sd(d+s)\cos(\frac{\ell s_1\pi}{d_1})>0$ if and only if $\cos(\frac{\ell s_1\pi}{d_1})>0$. In this case, $f(t)=d\cos(\frac{\ell s_1\pi}{d_1})$, and so (\ref{re}) yields
$\operatorname{Re}\left(U_{A(Y)}(t)_{v,v}\right)=\frac{1}{2}\big(1+\frac{d_1\cos(\ell s_1\pi/d_1)}{d_1+2s_1}\big)>\frac{1}{2}$.
\end{itemize}
Combining the above two cases with (\ref{dirprod3}) yields $\left|U_{A(K_2\times Y}(t)_{((u,v),(u,v)}\right|\geq C$ as desired.
\end{proof}

\begin{example} 
Let $v$ be an apex $v$ of $Y=O_2\vee Z$, where $Z$ is a $d$-regular graph on $n=\frac{1}{2}s(d+s)$ vertices.
\begin{enumerate}
\item Let $d=s$ so that $n=s^2$ and $d_1=s_1=1$. By Theorem \ref{k2xYsed}, $C=\min_{k\in\{1,2\}}\cos^2\left(\frac{k\pi}{3}\right)=\frac{1}{4}$. Thus, $\left|U_{A(K_2\times Y}(\tau)_{((u,v),(u,v)}\right|\geq C$ for all $t$ with equality at time $\tau=\frac{2k\pi}{d+2s}$, where $k\in\{1,2\}$.
\item Let $d=2s$ so that $n=\frac{3s^2}{2}$. In this case, $\nu_2(d)>\nu_2(s)$, and so PST occurs between the apexes of $Y$ at some time $\tau$ by \cite[Theorem 11(1)]{Kirkland2023}. This implies that $\operatorname{Re}\left(U_{A(Y)}(\tau)_{v,v}\right)=0$, and so by Theorem \ref{dirprod2}(2), we get that $(u,v)$ is not sedentary in $K_2\vee Y$ for any vertex $u$ of $K_2$.
\end{enumerate}
\end{example}

The conditions in Theorem \ref{k2xYsed} ensure that an apex $v$ of $Y=O_2\vee Z$ is sedentary \cite[Theorem 36(2)]{Monterde2023}. Thus, Theorem \ref{k2xYsed} implies that the bipartite double of a disconnected double cone  preserves the sedentariness of its apexes. In contrast, the bipartite double of a complete graph does not preserve the sedentariness of its vertices by Corollary \ref{kmxY1}(1).

We end this section with a remark about the Laplacian case. Note that the Laplacian matrix of $X\times Y$ is given by $L(X\times Y)=L(X)\otimes D(Y)+D(X)\otimes L(Y)-L(X)\otimes L(Y)$. If $X$ and $Y$ are both regular, then $L(X)\otimes D(Y)$, $D(X)\otimes L(Y)$ and $L(X)\otimes L(Y)$ pairwise commute, and so $U_{L(X\times Y)}(t)$ is simply the product of $e^{it(L(X)\otimes D(Y))}$, $e^{it(D(X)\otimes L(Y))}$ and $e^{it(L(X)\otimes L(Y))}$. In this case, $X\times Y$ is also regular and so the quantum walks relative to $A$ and $L$ are equivalent. However, if at least one of $X$ and $Y$ is not regular, then it is not clear how to obtain $U_{L(X\times Y)}(t)$. We leave this as an open question.

\section{Blow-ups}\label{secBU}

Throughout this section, we assume that $X$ is a simple weighted graph. A \textit{blow-up} of $m$ copies of $X$ is the graph with adjacency matrix
\begin{equation*}
    \textbf{J}_n\otimes A(X).
\end{equation*}
If $(j,u)$ denotes the $j$th copy of $u\in V(X)$ in $\up{m}X$, then $T_u=\{(j,u):j\in\mathbb{Z}_m\}$ is a twin set in $\up{m}X$. In particular, if $T$ is a twin set in $X$, then $\bigcup_{u\in T}T_u$ is a twin set in $\up{m}X$. It is known that every vertex of $\up{m}X$ is $(1-\frac{2}{m})$-sedentary for all $m\geq 3$ \cite[Theorem 24]{Monterde2023}. Here, we provide sharper bounds for vertex sedentariness in blow-upss and resolve the case $m=2$. First, we need to state \cite[Proposition 2]{bhattacharjya2023quantum}.
\begin{lemma}
\label{supp}
    If $u$ is a vertex of $X$, then
    \begin{equation}
        \sigma_{(j,u)}(A(\up{m}X))=\{m\lambda:\lambda\in\sigma_u(A(X))\}\cup\{0\}.
    \end{equation}
\end{lemma}

Suppose $X$ has $n$ vertices. If $\theta$ is an eigenvalue of $A(X)$ with associated orthogonal projection matrix $F_0$, then in $A(\up{m}X)$, \cite[Equation 2]{bhattacharjya2023quantum} implies that
\begin{center}
    $E_0=I_{mn}-\left[(1/m)\textbf{J}_m\otimes(I_n-F_0)\right]$.
\end{center}
Thus, if $0\in\sigma_u(A(X))$, then $(E_0)_{(j,u),(j,u)}=1-\frac{ 1-(F_0)_{u,u}}{m}=\frac{m-1}{m}+(F_{0})_{u,u}$. Now, if $\theta$ is not an eigenvalue of $A(X)$, then in $A(\up{m}X)$, \cite[Equation 5]{bhattacharjya2023quantum} implies that
\begin{center}
    $E_0=I_{mn}-\left[(1/m)\textbf{J}_m\otimes I_n\right]$.
\end{center}
Thus, if $0\notin\sigma_u(A(X))$, then $(E_0)_{(j,u),(j,u)}=\frac{m-1}{m}$. In both cases, taking $S=\{0\}$ and $a=(E_0)_{(j,u),(j,u)}$ in Theorem \ref{sed2} yields the following result.

\begin{theorem}
\label{size3}
Let $m\geq 2$ and $X$ be a simple weighted graph. The following hold relative to $A=A(\up{m}X)$.
\begin{enumerate}
    \item If $0\in\sigma_u(A(X))$, then $|U_{A}(t)_{(j,u),(j,u)}|\geq 1-\frac{2}{m}+2(F_0)_{u,u}$ for all $t$ and for each $j\in\mathbb{Z}_m$, where $F_0$ is the orthogonal projection matrix associated with the eigenvalue $0$ of $A(X)$.
    \item If $0\notin\sigma_u(A(X))$, then $|U_{A}(t)_{(j,u),(j,u)}|\geq 1-\frac{2}{m}$ for all $t$ and for each $j\in\mathbb{Z}_m$.    
\end{enumerate}
In both cases, equality holds if and only if there is a time $t_1$ such that $e^{it_1\lambda}=-1$ for all $0\neq \lambda\in\sigma_{(j,u)}(A)$.
\end{theorem}

If $m=2$ and $0\notin\sigma_u(A(X))$, Theorem \ref{size3}(2) implies that vertex $(j,u)$ may or may not be sedentary in $\up{2}X$. We resolve this in the next corollary, which is obtained by a direct application of Lemma \ref{eurekarem2}.

\begin{corollary}
\label{m=2}
    If $m=2$ and $0\notin\sigma_u(A(X))$, then $(j,u)$ is adjacency sedentary in $\up{m}X$ for each $j\in\mathbb{Z}_m$ if and only if there are integers $m_j$ such that $\sum_{\lambda_j\in\sigma_{u}(A(X))}m_j\lambda_j=0$ and $\sum_{ \lambda_j\in\sigma_{u}(A(X))}m_j$ is odd. If we add that $\phi(A(X),t)\in\mathbb{Z}[t]$ and $u$ is periodic, then the latter statement is equivalent to each eigenvalue $\lambda_j\in\sigma_{u}(A(X))$ is of the form $\lambda_j=b_j\sqrt{\Delta}$, where $b_j$ is an integer and either $\Delta=1$ or $\Delta>1$ is a square-free integer and the $\nu_2(b_j)$'s are not all equal. In this case, $u$ is tightly sedentary.
\end{corollary}

We close this section by providing families of graphs whose blow-ups contain sedentary vertices.

\begin{example}
Consider the cycle $C_n$ for all $n\geq 3$. By \cite[Theorem 8]{bhattacharjya2023quantum}, $\up{m}C_n$ does not admit PGST for all $m\geq 2$. Invoking Corollary \ref{yipee1}, we get that each vertex of $\up{m}C_n$ is sedentary for all $m\geq 2$.
\end{example}

\begin{example}
Consider the path $P_n$ for all $n\geq 3$. Making use of \cite[Theorem 8]{bhattacharjya2023quantum} and Corollary \ref{yipee1}, we get that vertex $(j,u)$ in $\up{m}P_n$ is sedentary for all $m\geq 2$ and for each all $j\in\mathbf{Z}_m$ whenever
\begin{itemize}
\item $n=2^tr-1,$ $t\geq 0$ and $r$ is an odd composite number; and
\item $n=2^tr-1,$ $t>0$, $r$ is an odd prime number and $u$ is not a multiple of $2^{t-1}$.
\end{itemize}
\end{example}

\section{Joins}\label{secJoins}

We now examine sedentariness under the join operation relative to $M\in\{A,L\}$. Here, we require the underlying graphs to be regular when dealing with $M=A$. To do this, we first need the following result due to Kirkland and Monterde \cite[Corollaries 70 and 72]{kirkland2023quantum}.

\begin{lemma}
\label{boundM}
Let $M\in\{A,L\}$. For all $u,v\in V(X)$ and for all $t$, we have
\begin{equation*}
\left|\ |U_{M(X\vee Y)}(t)_{u,v}|-|U_{M(X)}(t)_{u,v}|\ \right|\leq 2/|V(X)|.
\end{equation*}
\end{lemma}

\begin{theorem}
\label{joinsed}
If $u$ is $C$-sedentary in $X$ with $C>\frac{2}{|V(X)|}$, then $u$ is $\left(C-\frac{2}{|V(X)|}\right)$-sedentary in $X\vee Y$ for any graph $Y$, where we require that $X$ and $Y$ are both regular whenever $M=A$.
\end{theorem}
\begin{proof}
By assumption, we have $|U_X(t)_{u,u}|-2/|V(X)|\geq C-2/|V(X)|> 0$ for all $t$. Thus, Lemma \ref{boundM} yields $|U_M(t)_{u,u}|\geq |U_X(t)_{u,u}|-2/|V(X)|\geq C-2/|V(X)|>0$ for all $t$.
\end{proof}

From Theorem \ref{joinsed}, if $C$ is large enough, then $C$-sedentariness in $X$ is preserved in $X\vee Y$.

\begin{example}
Let $m,n$ be integers such that $\frac{mn}{m+n}>2$. By Corolllary \ref{kmxY1}(2), each vertex of $X=K_m\times K_n$ is $C$-sedentary, where $C=\frac{mn-2(m+n-1)}{mn}$. By assumption, $C>\frac{2}{mn}=\frac{2}{|V(X)|}$, and so each vertex of $X$ is $(\frac{mn-2(m+n)}{mn})$-sedentary in $X\vee Y$ for any graph $Y$ by Theorem \ref{joinsed}. Moreover, this holds for $M\in\{A,L\}$, where we require that $Y$ is regular whenever $M=A$.
\end{example}

If $C=2/|V(X)|$, then the conclusion in Theorem \ref{joinsed} may not hold for any graph $Y$.

\begin{example}
\label{38}
Let $k$ be odd and consider $X=CP(2k)$ with non-adjacent vertices $u$ and $v$. By Corollary \ref{cmg11}, every vertex in $X$ is tightly $C$-sedentary, where $C=\frac{1}{k}$. Let $Y$ be a graph on $n$ vertices that is $\ell$-regular whenever $M=A$. Note that $u$ and $v$ are twins in $X$ and in $X\vee Y$.
\begin{enumerate}
    \item Let $M=L$. From \cite[Example 41]{kirkland2023quantum}, PST occurs between $u$ and $v$ in $X\vee Y$ if and only if $n\equiv 2$ (mod 4). Since each vertex in $X\vee Y$ is periodic, PST and PGST are equivalent in $X\vee Y$. Hence, by Corollary \ref{yipee1}, each vertex of $X$ is sedentary in $X\vee Y$ if and only if $n\not\equiv 2$ (mod 4). 
    \item Let $M=A$. From \cite[Example 47]{kirkland2023quantum}, PST occurs between $u$ and $v$ in $X\vee Y$ if and only if $n=\frac{s(2k-\ell-2+s)}{2k}$ and $\nu_2(\ell)>\nu_2(s)=1$. Moreover, PGST occurs between $u$ and $v$ in $X\vee Y$ if and only if $(m+\ell-2)^2+4mn$ is not a perfect square. Invoking Corollary \ref{yipee1}, we get that each vertex of $X$ is sedentary in $X\vee Y$ if and only if $n=\frac{s(2k-\ell-2+s)}{2k}$ and $\nu_2(\ell)<\nu_2(s)$ or $\nu_2(s)\neq 1$.
\end{enumerate}
\end{example}

We end this section by showing a vertex in $X$ can be sedentary in $X\vee Y$ but not in $X$.

\begin{example}
Let $X=CP(2k)$, where $k$ is even. By Corollary \ref{cmg11}, no vertex in $X$ is sedentary since each is involved in PST with minimum time $\frac{\pi}{2}$. The same argument in Example \ref{38} yields the following:
\begin{enumerate}
    \item Let $M=L$. From \cite[Corollary 28]{kirkland2023quantum}, PST occurs between $u$ and $v$ in $X\vee Y$ if and only if $n\equiv 0$ (mod 4). Thus, each vertex of $X$ is sedentary in $X\vee Y$ if and only if $n\not\equiv 0$ (mod 4). 
    \item Let $M=A$. Applying \cite[Lemma 3(2)]{kirkland2023quantum} yields $\sigma_u(A(X\vee Y)=\{0,-2,\frac{1}{2}(2k-2+\ell\pm\sqrt{\Delta})\}$, where $\Delta=(2k-\ell-2)^2+8kn$. If $\Delta$ is a perfect square, then $u$ is involved in PST if and only if $\nu_2(2k-2+\ell\pm\sqrt{\Delta})=1$ \cite[Theorem 10]{Kirkland2023}. But if $\Delta$ is not a perfect square, then $u$ is involved in PGST if and only if $2k-2+\ell\equiv 0$ (mod 4) \cite[Theorem 13]{Kirkland2023}. Thus, each vertex of $X$ is sedentary in $X\vee Y$ if and only if either (i) $\Delta$ is a perfect square and $\nu_2(\lambda)\neq 1$ for some $\lambda\in \{2k-2+\ell\pm\sqrt{\Delta}\}$ or (ii) $2k-2+\ell\not\equiv 0$ (mod 4).
\end{enumerate}
\end{example}

\section{Open questions}\label{secOQ}

In this paper, we proved new results on sedentariness. We showed that sedentariness and PGST are dichotomous amongst twin vertices. This allowed us to characterize sedentary vertices in complete multipartite graphs and threshold graphs and provide sharp bounds on their sedentariness. We also investigated sedentariness under direct products, blow-ups and joins. Our results can be used to constructs new examples of graphs with sedentary vertices. We now present some open questions that deserve further exploration. %To inspire further study.
\begin{enumerate}
\item Characterize sedentariness in trees, Cayley graphs and distance-regular graphs.
\item Under what conditions is a vertex in a graph neither sedentary nor involved in PGST? Are there other types of vertices that admit the dichotomous property similar to that of twin vertices?
    \item Is the addition of a weighted loop or an attachment of a pendent path to a vertex induce or preserve sedentariness? If yes, then we ask: which weights of loops or lengths of paths achieve this task?
    \item Determine other graph operations (such as the rooted product, lexicographic product, strong product and corona product) that induce and/or preserve sedentariness.
    \item Characterize sedentary vertices in threshold graphs relative to the adjacency matrix, and provide tight bounds on their sedentariness.
\end{enumerate}

\section*{Acknowledgements}
I thank the University of Manitoba Faculty of Science and Faculty of Graduate Studies for the support. I also thank Steve Kirkland and Sarah Plosker for the guidance.

\bibliographystyle{alpha}
\bibliography{mybibfile}

% \bib, bibdiv, biblist are defined by the amsrefs package.
\begin{bibdiv}
\begin{biblist}

\bib{Alvir2016}{article}{
      author={Alvir, Rachael},
      author={Dever, Sophia},
      author={Lovitz, Benjamin},
      author={Myer, James},
      author={Tamon, Christino},
      author={Xu, Yan},
      author={Zhan, Hanmeng},
       title={{Perfect state transfer in Laplacian quantum walk}},
        date={2016},
        ISSN={15729192},
     journal={Journal of Algebraic Combinatorics},
      volume={43},
      number={4},
       pages={801\ndash 826},
         url={https://arxiv.org/abs/1409.5840},
}

\bib{bhattacharjya2023quantum}{article}{
      author={Bhattacharjya, Bikash},
      author={Monterde, Hermie},
      author={Pal, Hiranmoy},
       title={Quantum walks on blow-up graphs},
        date={2023},
     journal={arXiv preprint arXiv:2308.13887},
}

\bib{coutinho2016perfect}{article}{
      author={Coutinho, Gabriel},
      author={Godsil, Chris},
       title={Perfect state transfer in products and covers of graphs},
        date={2016},
     journal={Linear and Multilinear Algebra},
      volume={64},
      number={2},
       pages={235\ndash 246},
}

\bib{Coutinho2014}{article}{
      author={Coutinho, Gabriel},
       title={{Quantum state transfer in graphs}},
        date={2014},
     journal={PhD. Dissertation. UWSpace},
}

\bib{SedQW}{article}{
      author={Godsil, Chris},
       title={Sedentary quantum walks},
        date={2021},
        ISSN={0024-3795},
     journal={Linear Algebra and its Applications},
      volume={614},
       pages={356\ndash 375},
  url={https://www.sciencedirect.com/science/article/pii/S0024379520304067},
        note={Special Issue ILAS 2019},
}

\bib{HardyAsymptoticFI}{article}{
      author={Hardy, Gordon~H.},
      author={Ramanujan, Srinivasa},
       title={Asymptotic formulae in combinatory analysis},
        date={1918},
     journal={Proceedings of The London Mathematical Society},
      volume={17},
      number={2},
       pages={75\ndash 115},
         url={https://api.semanticscholar.org/CorpusID:121252456},
}

\bib{kirkland2023quantum}{article}{
      author={Kirkland, Steve},
      author={Monterde, Hermie},
       title={Quantum walks on join graphs},
        date={2023},
     journal={arXiv preprint arXiv:2312.06906},
}

\bib{Kirkland2023}{article}{
      author={Kirkland, Steve},
      author={Monterde, Hermie},
      author={Plosker, Sarah},
       title={Quantum state transfer between twins in weighted graphs},
        date={2023},
     journal={Journal of Algebraic Combinatorics},
      volume={58},
      number={3},
       pages={623\ndash 649},
         url={https://doi.org/10.1007/s10801-023-01261-3},
}

\bib{Severini}{article}{
      author={Kirkland, Steve},
      author={Severini, Simone},
       title={Spin-system dynamics and fault detection in threshold networks},
        date={2011},
     journal={Physical Review A},
      volume={83},
       pages={012310},
         url={https://link.aps.org/doi/10.1103/PhysRevA.83.012310},
}

\bib{Kirkland2020}{article}{
      author={Kirkland, Steve},
      author={Zhang, Xiaohong},
       title={{Fractional Revival of Threshold Graphs under Laplacian
  Dynamics}},
        date={2020},
        ISSN={20835892},
     journal={Discussiones Mathematicae - Graph Theory},
      volume={40},
      number={2},
       pages={585\ndash 600},
}

\bib{Monterde2022}{article}{
      author={Monterde, Hermie},
       title={{Strong cospectrality and twin vertices in weighted graphs}},
        date={2022},
        ISSN={10813810},
     journal={Electronic Journal of Linear Algebra},
      volume={38},
       pages={494\ndash 518},
}

\bib{Monterde2023}{article}{
      author={Monterde, Hermie},
       title={Sedentariness in quantum walks},
        date={2023},
     journal={Quantum Information Processing},
      volume={22},
      number={7},
       pages={273},
         url={https://doi.org/10.1007/s11128-023-04011-3},
}

\end{biblist}
\end{bibdiv}
\end{document}